\documentclass[10pt]{article}
\usepackage{amsmath}
\usepackage{amsfonts}
\usepackage{mathrsfs}
\usepackage{bbm}
\usepackage{amsfonts}
\usepackage{enumitem}
\usepackage{amssymb,amsmath,amsthm,graphicx}
\usepackage{float} \usepackage[colorlinks=true]{hyperref}

\hypersetup{urlcolor=blue, citecolor=red}
\def\Xint#1{\mathchoice
{\XXint\displaystyle\textstyle{#1}}%
{\XXint\textstyle\scriptstyle{#1}}%
{\XXint\scriptstyle\scriptscriptstyle{#1}}%
{\XXint\scriptscriptstyle\scriptscriptstyle{#1}}%
\!\int}
\def\XXint#1#2#3{{\setbox0=\hbox{$#1{#2#3}{\int}$ }
\vcenter{\hbox{$#2#3$ }}\kern-.6\wd0}}

\def\dashint{\Xint-}

\allowdisplaybreaks

\newcommand{\be}{\begin{equation} \label}
\newcommand{\ee}{\end{equation}}
\newcommand{\bea}{\begin{eqnarray}\label}
\newcommand{\eea}{\end{eqnarray}}
\newcommand{\bas}{\begin{eqnarray*}}
\newcommand{\eas}{\end{eqnarray*}}
\newcommand{\bit}{\begin{itemize}}
\newcommand{\eit}{\end{itemize}}

\newcommand{\lbal}{\left\{ \begin{array}{l}}
\newcommand{\lball}{\left\{ \begin{array}{ll}}
\newcommand{\ear}{\end{array} \right.}

\begin{document}
\title{Partial regularity of suitable weak solutions to the incompressible magnetohydrodynamic equations}

\author{Mengyao Ding$^{1}$,~~Wenwen Huo$^{2}$,~~Chao Zhang$^{2,1}$\thanks{Corresponding author.
E-Mail:
mengyaod@126.com (M. Ding),
huowenwen2023@163.com (W. Huo),
 czhangmath@hit.edu.cn (C. Zhang).}\\
 {\small $^{1}$ Institute for Advanced Study in Mathematics, Harbin Institute of Technology, Harbin 150001, PR China}\\
{\small $^{2}$ School of Mathematics, Harbin Institute of Technology, Harbin 150001, PR China}\\
}
\date{}
\newtheorem{theorem}{Theorem}
\newtheorem{definition}{Definition}[section]
\newtheorem{lemma}{Lemma}[section]
\newtheorem{proposition}{Proposition}[section]
\newtheorem{corollary}{Corollary}\newtheorem{remark}{Remark}
\renewcommand{\theequation}{\thesection.\arabic{equation}}
\catcode`@=11 \@addtoreset{equation}{section} \catcode`@=12
\maketitle{}

\begin{abstract}

This paper establishes a regularity theory for the magnetohydrodynamics (MHD) equations with external forces through scaling analysis. Inspired by the existing methodology, we utilize linearized approximations and the monotonicity property of harmonic functions to construct iterative sequences capturing scaling properties. This work successfully extends Navier-Stokes techniques to MHD coupling and demonstrates that the one dimensional parabolic Hausdorff measure of the possible singular points for the suitable weak solutions is zero.

\begin{description}
\item[2020MSC:] 35Q35; 76B03; 76W05
\item[Keywords:] Magnetohydrodynamic equations; Suitable weak solutions; Partial regularity
\end{description}
\end{abstract}
\section{Introduction}

In this paper, we investigate the following incompressible magnetohydrodynamic (MHD) equations
\begin{align}\label{equ1-1}
\left\{
  \begin{array}{ll}
    \partial_{t}u-\Delta u+(u\cdot\nabla)u-(B\cdot\nabla)B+\nabla \left(p+\frac{1}{2}|B|^{2}\right)=f,\\
    \partial_{t}B-\Delta B+(u\cdot\nabla)B-(B\cdot\nabla)u=g,\\
    {\rm div}\,u=0,\quad {\rm div}\,B=0,
     \end{array}
\right.
\end{align}
where $u$, $p$ and $B$ represent the velocity field, its pressure and the magnetic field, respectively.
Here, we introduce the modified pressure $P:=p+\frac{1}{2}|B|^{2}$ for notational simplicity.
In \eqref{equ1-1}, the terms $f$ and $g$ account for external forces, with $f$ denoting gravitational or mechanical force
and $g$ describing external magnetic excitation such as current density. 
As regards $f\in L^{q}(\Omega \times I)$ with some $q\in(1, \infty)$, 
we may assume $\operatorname{div} f = 0$ distributionally without loss of generality. Indeed,
the Helmholtz decomposition allows us to write $f = \nabla \Phi + f'$ with $\operatorname{div} f' = 0$ and $\|f'\|_{L^q(\Omega \times I)} \leq C(q, \Omega\times I) \|f\|_{L^q(\Omega \times I)}$, and the gradient term $\nabla \Phi$ can be absorbed into $\nabla P$. 

In the absence of the magnetic field, the system \eqref{equ1-1} under the force-free situation reduces to the incompressible
Navier-Stokes equations 
\begin{align}\label{equ1-1'}
\left\{
  \begin{array}{ll}
    \partial_{t}u-\Delta u+(u\cdot\nabla)u+\nabla P=0,\\
    {\rm div}\,u=0.
     \end{array}
\right.
\end{align}
In the seminal work \cite{CKN1982}, Caffarelli, Kohn and Nirenberg proved that the suitable weak solutions of \eqref{equ1-1'} remain locally regular
as long as 
\begin{align}\label{gu}
\limsup_{r\rightarrow 0}\frac{1}{r}\int_{Q_{r}}|\nabla u|^{2}\,dxdt\le \epsilon_0
\end{align}
with a certain absolute constant $\epsilon_0>0$. This local regularity criteria implies that solutions 
are smooth away from a set with a null Hausdorff measure of dimension one.
The original proof in \cite{CKN1982} employs a series of delicate iterations, 
which have been simplified by Lin \cite{Lin1998}, Ladyzhenskaya and Seregin \cite{LS1999}, among others.
Alternative local regularity criteria replacing \eqref{gu} have also been developed; 
for instance, Tian and Xin \cite{TX1999} showed that smallness of either the scaled local $L^2$-norm of 
vorticity or the scaled local energy is sufficient to derive smoothness in a neighborhood.
Very recently, Wang \cite{Wang2025} introduced a novel geometric framework to prove the celebrated Caffarelli-Kohn-Nirenberg theorem, 
primarily utilizing a new compactness lemma and monotonicity properties of harmonic functions. 

Let us review some related literature concerning the MHD equations without external forces (i.e., $f\equiv0$ and $g\equiv0$).
 Building upon the approach developed for the Navier-Stokes equations, He and Xin \cite{HX2005} 
 established interior regularity under suitable local smallness conditions and suggested that the velocity field plays a more dominant role than the magnetic field.
Besides, an alternative regularity condition was presented by Kang and Lee in \cite{KL2009}, 
where the authors required the smallness of a scaled norm of the velocity $u$, coupled with the boundedness (not necessarily smallness) 
of a corresponding norm for the magnetic field $B$.  The observation that smallness of $u$ alone might provide a desired control of the full solution $(u, B)$ was further pursued by Wang and Zhang \cite{WZ2013}.
The authors determined a partial regularity condition relying exclusively on the scaled norm of the velocity $u$;
more specifically, they proved the existence of $r_0>0$ and $\varepsilon_0>0$ such that the condition 
\begin{align*}
\sup_{0<r<r_{0}}r^{1-\frac{3}{p}-\frac{2}{q}}\left\|u\right\|_{L^{p,q}(Q_{r})}\le\epsilon_0 \quad{\rm with}~~1\leq \frac{3}{p}+\frac{2}{q} \leq 2
\end{align*} 
suffices to guarantee the local regularity of suitable weak solutions.
Beyond the aforementioned results closely related to our motivation, 
various other regularity criteria for both Navier-Stokes and MHD systems have been extensively studied; see, e.g., 
\cite{CW2010,CY2015,CY2016,CMZ2008,CQZ2022,DD2007,GKT2007,H2013,JW2014,KK2012,KK2014,LW2018,MNS2007,S1988,V2007,W2008,WWZ2016,Z2007}.

Incorporating the external force term into the MHD equations significantly enhances the theory's capability to describe complex physical scenarios.
 In the presence of external forces, Chamorro and He \cite{2CH2021} reformulated the MHD equations into a symmetric form by 
adopting the Elsasser formulation, and then established regularity criteria by extending the Caffarelli-Kohn-Nirenberg theory to the MHD equations within a parabolic Morrey-space framework.
Furthermore, Chamorro and He \cite{1CH2021} showed that outside a set of vanishing parabolic one-dimensional Hausdorff measure, 
the velocity field $u$ and the magnetic field  $B$ are locally regular provided $f, g\in L^2((0, T); H^1(\Omega))$.  
Under some local boundedness assumptions, a gain of regularity for weak solutions to the MHD equations was obtained by 
Chamorro, Cortez, He and Jarr\'{i}n \cite{CCHJ2021} in the general setting of the Serrin regularity theory.
In addition, for the case $g=0$, Ren, Wang and Wu \cite{RWW2016} studied the partial regularity of the suitable weak solutions
to the fractional MHD equations in $\mathbb{R}^n$ for $n = 2, 3$, showing a non-dominant role of the magnetic field for the local regularity.

The purpose of this paper is to study the Caffarelli-Kohn-Nirenberg type regularity theory for the incompressible MHD equation \eqref{equ1-1} 
in the presence of external forces $f, g \in L^q$ ($q > \frac{5}{2}$).  Our approach builds upon the geometric framework recently developed by Wang \cite{Wang2025} for the Navier-Stokes equations. 
By adapting this framework and systematically tackling the new challenges posed by the magnetic field $B$ as well as the external forces $f$ and $g$, 
we successfully extend the classical partial regularity theory to the forced MHD system.  Moreover, combining our main result with the strategy pioneered by Wang and Zhang \cite{WZ2013}, we obtain a regularity criteria (stated in Corollary \ref{pro1-1}) 
that requires only the norm of the velocity field $u$ to be small, while the external forces merely need to satisfy a mild integrability condition.
\\[5pt]
\noindent{\bf Notations.}
Following the common usage, we denote $B_{r}(x)$ by a ball with radius $r>0$ centered at $x\in \Omega$.
We let $Q_{r}(x,t)$  be a cylinder given by $Q_{r}(x,t)=B_{r}(x)\times(t-r^{2},t)$ with $(x,t)\in \Omega\times (0,T)$.
When there is no risk of ambiguity, we will omit the center and simplify the notations by $B_r = B_r(x)$ and $Q_r=Q_{r}(x,t)$.
For an integrable function $\varphi$, the integral averages of $\varphi$ over $B_{r}$ and $Q_{r}$ are defined as 
\begin{align*}
\dashint_{B_{r}} \varphi\,dx=\frac{1}{|B_{r}|} \int_{B_{r}} \varphi\,dx,\quad
\dashint_{Q_{r}} \varphi\,dxdt=\frac{1}{|Q_{r}|} \int_{Q_{r}} \varphi\,dxdt.
\end{align*}
And for exponents $1\leq p,q\leq\infty$, the space $L^{p,q}(Q_{r})$ is equipped with the norm
\begin{align*}
\left\|\varphi\right\|_{L^{p,q}(Q_{r})} := 
\left(\int_{t-r^2}^t \left( \int_{B_r} |\varphi(y,s)|^p \,dy \right)^{q/p} \,ds \right)^{1/q}, \quad 1\leq p, q < \infty, 
\end{align*}
where the corresponding integral is replaced by the essential supremum in the usual way in the case of $p=\infty$ or $q=\infty$.
In particular, the space $L^{p,p}(Q_{r})$ reduces to the standard Lebesgue space $L^{p}(Q_{r})$.

Given a set $X\subset \mathbb{R}^{3}\times\mathbb{R}$, for any $s\geq0$, the $s$-dimensional parabolic Hausdorff measure 
is defined by  
\begin{align*}
P^{s}(X)=\lim_{\delta\rightarrow 0^{+}}P^{s}_{\delta}(X),
\end{align*}
where $P^{s}_{\delta}$ is a $\delta$-approximation obtained by setting 
\begin{align*}
P^{s}_{\delta}(X)=\inf \left\{\sum_{i} r_{i}^{s}:X\subset \bigcup_{i}Q_{r_{i}}: r_{i}<\delta\right\}.
\end{align*}
In particular, $P^{s}(X)=0$ if and only if for all $\epsilon>0$, the set $X$ can be covered by a collection 
$\{Q_{r_{i}}\}$ such that $\sum_{i}r_{i}^{s}<\epsilon$.

Finally, we use the symbol $C$ to denote a universal constant without emphasizing specific value and may change line by line.
\\[5pt]
\noindent{\bf Main results.}
Our results are established for suitable weak solutions, where the corresponding solution concept for the MHD equations \eqref{equ1-1} 
is defined as below.
\begin{definition}\label{def1-1}
Let $\Omega\subset \mathbb{R}^{3}$ and $T>0$. The triple $(u,B,P)$ is called a suitable weak solution to the MHD equations \eqref{equ1-1} in 
$\Omega\times(-T,0)$ if the following conditions are satisfied:
\begin{itemize}
	\item [\rm (a)] $u, B\in L^{\infty}\left((-T,0);L^2(\Omega)\right)\cap L^{2}\left((-T,0);H^1(\Omega)\right)$, 
	$P\in L^{\frac{3}{2}}\left((-T,0),L^{\frac{3}{2}}(\Omega)\right)$;
	
	\item [\rm (b)] $(u,B,P)$ solves the MHD equations in $\Omega\times(-T,0)$ in the sense of distribution;
	
	\item [\rm (c)] $(u,B,P)$ satisfies the following inequality
	\begin{align}\label{equ1-2}
	&\int_{\Omega}\left(|u(x,t)|^{2}+|B(x,t)|^{2}\right)\phi(x,t) \, dx
	+2\int_{-T}^{t}\int_{\Omega}\left(|\nabla u |^{2}+|\nabla B |^{2}\right)\phi \,dxds \nonumber\\
	\leq&\int_{-T}^{t}\int_{\Omega}\left(|u|^{2}+|B|^{2}\right)\left(\partial_{s}\phi+\Delta\phi\right)\,dxds
	+\int_{-T}^{t}\int_{\Omega}\left(u\cdot\nabla\phi\right)\left(|u|^{2}+|B|^{2}+2P\right)\,dxds \\
	&-2\int_{-T}^{t}\int_{\Omega}\left(B\cdot\nabla\phi\right)\left(u\cdot B\right)\,dxds
	+2\int_{-T}^{t}\int_{\Omega}\phi\left(u\cdot f+B\cdot g\right)\,dxds,\quad{\rm a.e.} ~t\in(-T,0) \nonumber
	\end{align}
	for each $\phi\geq 0$ in $C_{0}^{\infty}\left(\Omega\times(-T,0)\right)$.
\end{itemize}

\end{definition}

To address disturbances posed by the external force terms, 
we introduce the functionals for $f$ and $g$:
\begin{align}\label{F}
F_{m,r}\left[f,g\right]:=r^{2}\sqrt[m]{\dashint_{Q_{r}}|f|^{m}\,dxdt}+r^{2}\sqrt[m]{\dashint_{Q_{r}}|g|^{m}\,dxdt},\quad m >\frac{5}{2}.
\end{align}
Adopting the method from \cite{Wang2025,Lin1998}, we invoke the cubical functional 
\begin{align}\label{Cr}
C_{r}\left[u,B,P\right]:=\sqrt[3]{\dashint_{Q_{r}}|u|^{3}\,dxdt}+\sqrt[3]{\dashint_{Q_{r}}|B|^{3}\,dxdt}+r\sqrt[\frac{3}{2}]{\dashint_{Q_{r}}|P|^{\frac{3}{2}}\,dxdt},
\end{align}
which serves to capture the scaled properties of the suitable weak solution $(u, B, P)$.
The local regularity of solutions can be established by imposing smallness conditions on $C_{1}\left[u, B, P\right]$ and $F_{m,1}\left[f,g\right]$.

We are now in a position to present our main results.

\begin{theorem}\label{the1-1}
Let  $(u, B, P)$ be a suitable weak solution of \eqref{equ1-1} in $Q_1$, where $f, g\in L^{m}(Q_1)$ with $m>\frac{5}{2}$. Then there exists a small positive constant 
$\delta_ {*}=\delta_{*}(m)$ such that if
\begin{align}\label{assump1}
C_{1}\left[u,B,P\right]
+\sqrt{F_{m,1}\left[f,g\right]}\leq \delta_{*},
\end{align}
then we have
\begin{align*}
u,B\in C^{\alpha}(\overline{Q}_{\frac{1}{2}})
\end{align*}
for some small $\alpha\in(0,1)$.
\end{theorem}

\begin{remark}\label{r2}
Our proof of Theorem \ref{the1-1} adapts the framework from \cite{Wang2025} for the Navier-Stokes equations, 
where a key quantity $C_r\left[u, P\right]$ was introduced to perform an iteration scheme. 
For the MHD system \eqref{equ1-1}, our analysis successfully reveals that the coupled magnetic field component $B$ does not 
fundamentally compromise the applicability of the existing framework. For handling the disturbance of external force terms,
a delicate analysis is required to guarantee that our regularity criterion is established on the mildest regularity requirements on $f$ and $g$. 
To achieve this, we introduce the functional $F_{m,r}\left[f,g\right]$, designed to admit the necessary decay properties as $r \rightarrow 0$. 
Crucially, its square root $\sqrt{F_{m,r}\left[f,g\right]}$ combined with $C_r\left[u,B,P\right]$ implements iteration procedure,
constructing proper vectors $V_{k}$, $M_{k}$ and time-dependent functions $h_{k}(t)$ satisfying
\begin{align}\label{equ1-4}
C_{r^{k}}\left[u-V_{k}, B-M_{k},P-h_{k}(t)\right]
\leq \frac{1}{2^{k}}C_{1}\left[u,B,P\right]
+\frac{k}{2^{k}}\sqrt{F_{m,1}\left[f,g\right]}, \quad \forall ~k\in \Bbb{N}^+.
\end{align}
Here, the differences $V_{k+1}-V_k$ and $M_{k+1}-M_k$ diminish as $k$ increases. More specifically, 
\begin{align*}
\max\big\{|V_{k+1}-V_{k}|\, ,\,|M_{k+1}-M_{k}|\big\}
\leq& \frac{\Lambda}{2^{k+1}}C_{1}\left[u,B,P\right]
+\frac{\Lambda(k+1)}{2^{k+1}} \sqrt{F_{m,1}\left[f,g\right]}, \quad \forall ~k\in \Bbb{N}^+,
\end{align*}
where $\Lambda$ is a universal constant. These oscillation estimates above directly lead to the local H\"{o}lder continuity asserted in Theorem  \ref{the1-1}.
Thus, while the outer scaffolding of the iteration is inherited from \cite{Wang2025}, 
its internal mechanics are redesigned to accommodate the coupled magnetic field and the external force terms.

\end{remark}

\begin{theorem}\label{the1-2}
Let $(u, B, P)$ be a suitable weak solution of \eqref{equ1-1} in $Q_{1}$. Suppose that $f$ and $g$ belong to $L^{m}(Q_{1})$ with $m>\frac{5}{2}$. 
There exists a small positive constant $\delta_{**}$ such that for some $R>0$, if
\begin{align}\label{equ1-5}
\sup_{0<r\leq R}\frac{1}{r}\int_{Q_{r}}\left(|\nabla u|^{2}+|\nabla B|^{2}\right)dxdt\leq \delta_{**},
\end{align}
then $u$, $B$ are regular near the origin. Furthermore, if $\mathcal{S}$ denotes the singular set for any suitable 
weak solution of the MHD equations, then the one-dimensional parabolic Hausdorff measure of $\mathcal{S}$ is zero.
\end{theorem}

As established in \cite{WZ2013}, for the MHD equations without external forces, the velocity field plays a more decisive role than the magnetic field in regularity criteria. This naturally gives rise to the question of whether such velocity dominance persists in the presence of external forces. Through the subsequent corollary, we affirm this phenomenon by simultaneously combining Theorem \ref{the1-2} with \cite[Theorem 1.1]{WZ2013}.

\begin{corollary} \label{pro1-1}
Let $(u, B, P)$ be a suitable weak solution of \eqref{equ1-1} in $Q_{1}$.
Suppose that $f$ and $g$ belong to $L^{m}(Q_{1})$ with $m>\frac{5}{2}$. 
Then there exist $\overline{R}>0$ and a small positive constant $\overline{\delta}_{**}$ such that if $u \in L^{p,q}(Q_{\overline{R}})$ and
\begin{align*}
\sup_{0<r<\overline{R} }r^{1-\frac{3}{p}-\frac{2}{q}}\left\|u\right\|_{L^{p,q}(Q_{r})}<\overline{\delta}_{**}
\end{align*}
with $(p, q)$ satisfying
\begin{align*}
1\leq \frac{3}{p}+\frac{2}{q} \leq 2,\quad 1\leq q\leq\infty,\quad (p,q)\neq (\infty,1),
\end{align*}
then $u$ and $B$ are regular near the origin.
\end{corollary}

\noindent
{\bf Structure of the present work.}
Section \ref{sec2} collects key preliminary lemmas. 
In Section \ref{sec3}, motivated by the works of \cite{Wang2025, Lin1998}, 
we analyze the scaling properties of the MHD equations through a linearization-based approximation scheme. 
This approach establishes the Caffarelli-Kohn-Nirenberg theory for the equations \eqref{equ1-1}, as stated in Theorem \ref{the1-1}.
Building on the result of Theorem \ref{the1-1}, Section \ref{sec4} is dedicated to proving Theorem \ref{the1-2} and Corollary \ref{pro1-1}, 
which develop the partial regularity theory under an alternative regularity condition.

\section{Preliminaries}\label{sec2}

As essential ingredients for the  the subsequent iterative analysis,
we employ monotonicity formulas  adapted from \cite{Wang2025}. 
These formulas can be obtained via applications of the H\"{o}lder inequality, Poincar\'{e} inequality and the classical 
monotonicity inequality for harmonic functions; see the detailed proof in \cite[Lemma 2 and Lemma 3]{Wang2025}.

\begin{lemma} [\label{lem2-3}Inhomogeneous Monotonicity]
{\rm (a)} For any $1\leq q<3$ and $1\leq p\leq p^{*}:=\frac{3q}{3-q}$, 
there are universal constants $C_{1}$ and $C_{2}$ such that for any weak solution $u$ in $B_{1}\subset\mathbb{R}^3$ of
\begin{align*}
\Delta u={\rm div}\,F,
\end{align*}
we have for any $0<r\leq 1$:
\begin{align}\label{equ2-3}
\dashint_{B_{r}}|u|^{p}\,dx
\leq C_{1}\dashint_{B_{1}}|u|^{p}\,dx+\frac{C_{2}}{r^{3}}\Big(\dashint_{B_{1}}|F|^{q}\,dx\Big)^{\frac{p}{q}}.
\end{align}
{\rm (b)}  For any $1< p<\infty$, there are universal constants $C_{1}$ and $C_{2}$ such that 
for any $G=\{G_{ij}\}$, $G_{ij}\in L^{p}(B_{1})$ and any distribution solution $u$ in $B_{1}\subset\mathbb{R}^3$ of
\begin{align*}
\Delta u={\rm div}^{2}\,G:=\sum_{i,j=1}^{3}\left(G_{ij}\right)_{x_{i}x_{j}},
\end{align*}
we have for any $0<r\leq 1$:
\begin{align}\label{equ2-2}
\dashint_{B_{r}}|u|^{p}\,dx
\leq C_{1}\dashint_{B_{1}}|u|^{p}\,dx+\frac{C_{2}}{r^{3}}\dashint_{B_{1}}|G|^{p}\,dx.
\end{align}
\end{lemma}

\begin{lemma} [\label{lem2-4}Interpolated Monotonicity]
There are universal constants $C_{1}$ and $C_{2}$ such that for any $u$ in $H^{1}(B_{1})$ for $B_{1}\subset\mathbb{R}^3$, we have for $0<r\leq 1$:
\begin{align}\label{equ2-3}
\dashint_{B_{r}}|u|^{3}\,dx
\leq C_{1}\dashint_{B_{1}}|u|^{3}\,dx+\frac{C_{2}}{r^{3}}\Big(\int_{B_{\frac{3}{4}}}|u|^{2}\,dx\Big)^{\frac{1}{2}}\int_{B_{1}}|\nabla u|^{2}\,dx.
\end{align}
\end{lemma}

Finally, we come back to the MHD equations and provide a fundamental local energy estimate of  suitable weak solutions.
\begin{lemma}\label{lem2-5}
Let $(u,B,P)$ be a suitable weak solution of \eqref{equ1-1} in $Q_{1}$. Suppose that $f$ and $g$ belong to $L^{m}(Q_{1})$ with $m>\frac{5}{2}$. 
Then there exists a constant $C>0$ such that
\begin{align*}
&\sup_{t\in(-\frac{9}{16},0)}\int_{B_{\frac{3}{4}}}\left(|u(x,t)|^{2}+|B(x,t)|^{2}\right)\,dx
+2\int_{Q_{\frac{3}{4}}}\left(|\nabla u|^{2}+|\nabla B|^{2}\right)\,dxds \nonumber\\
\leq& C\int_{Q_{1}}\left(|u|^{3}+|B|^{3}+|P|^{\frac{3}{2}}\right)\,dxds
+C\left(\int_{Q_{1}}\left(|u|^{3}+|B|^{3}\right)\,dxds\right)^{\frac{1}{3}}
+C\int_{Q_{1}}\left(|f|^{m}+|g|^{m}\right)\,dxds.
\end{align*}
\end{lemma}

\begin{proof}
By selecting a proper test function $\phi(x,t)\in C_{0}^{\infty}(Q_{1})$ in \eqref{equ1-2} with the following properties:
\begin{align*}
\phi= 1 \quad{\rm in}~Q_{\frac{3}{4}}
\quad{\rm and}\quad 
0\leq \phi\leq 1,~|\nabla\phi|\leq C, ~|\partial_{s}\phi+\Delta\phi|\leq C\quad {\rm in}~Q_1,
\end{align*}
we obtain the desired inequality through applications of the Young inequality and the H\"{o}lder inequality.
The detailed computations are omitted here.
\end{proof}
\section{Regularity in the Linear Scale}\label{sec3}

In this section, we are devoted to proving Theorem \ref{the1-1}
 by analyzing the scaling properties of the linearized MHD equations.
Let us first recall the definitions of $F_{m,r}\left[f,g\right]$ and $C_{r}\left[u,B,P \right]$ from \eqref{F} and \eqref{Cr}:
\begin{align}\label{F'}
F_{m,r}\left[f,g\right]:=r^{2}\sqrt[m]{\dashint_{Q_{r}}|f|^{m}\,dxdt}
+r^{2}\sqrt[m]{\dashint_{Q_{r}}|g|^{m}\,dxdt},\quad m >\frac{5}{2},
\end{align}
and
\begin{align}\label{Cr'}
C_{r}\left[u,B,P\right]:=\sqrt[3]{\dashint_{Q_{r}}|u|^{3}\,dxdt}
+\sqrt[3]{\dashint_{Q_{r}}|B|^{3}\,dxdt}
+r\sqrt[\frac{3}{2}]{\dashint_{Q_{r}}|P|^{\frac{3}{2}}\,dxdt}.
\end{align}
If $(u,B)$ is regular near the origin, there exist constant vectors $\hat{V}$ and $\hat{M}$
such that $C_{r}\left[u-\hat{V},B-\hat{M},0\right]$ $\rightarrow0$ as $r\rightarrow0$. 
Based on this observation, our goal is to select two families of vector sequences $\{V_k\}_{k\ge1}$ and  $\{M_k\}_{k\ge1}$ such that
within the $r^k$-scaled domain, the quantity $C_{r^k}\left[u-V_{k},B-M_{k},0\right]$ can be controlled by $r^{k\alpha}$ with some order $\alpha>0$.
Following the strategy for the Navier-Stokes counterpart in \cite[Section 4]{Wang2025}, we also achieve our objective via an iterative procedure.
The central step is established by Lemma \ref{lem3-1} and Proposition \ref{pro3-1}.

We initialize the iteration by defining the scaled solutions:
\begin{align}\label{scaled}
u_{1}(x,t)=u(rx,r^{2}t)-V_{1},\quad
B_{1}(x,t)=B(rx,r^{2}t)-M_{1},\quad
P_{1}(x,t)=rP(rx,r^{2}t)-rh_{1}(r^{2}t),
\end{align}
with corresponding scaled external forces:
\begin{align}\label{scaled'}
f_{1}(x,t)=r^{2}f(rx,r^{2}t),\quad
g_{1}(x,t)=r^{2}g(rx,r^{2}t),
\end{align} 
where $V_{1}$ and $M_{1}$ are constant vectors to be specified later and $h_{1}(t)$ is a time-dependent function. 
For our iterative scheme, we consider the following equations:
\begin{align}\label{equ3-1}
\left\{
  \begin{array}{ll}
    \partial_{t}u_{1}-\Delta u_{1}+(rV_{1}+ru_{1})\cdot\nabla u_{1}-(rM_{1}+rB_{1})\cdot\nabla B_{1}+\nabla P_{1}=f_{1}   &\quad {\rm in}~ Q_{1},\\
    \partial_{t}B_{1}-\Delta B_{1}+(rV_{1}+ru_{1})\cdot\nabla B_{1}-(rM_{1}+rB_{1})\cdot\nabla u_{1}=g_{1}                &\quad {\rm in}~ Q_{1},\\
    {\rm div}\,u_{1}=0,\quad {\rm div}\,B_{1}=0                                                                               &\quad {\rm in}~ Q_{1}. 
     \end{array}
\right.
\end{align}
Note that the original equations \eqref{equ1-1} for $(u,B,P)$ can be recovered from \eqref{equ3-1} by setting $r=1$ and $V_1 = M_1 = 0$.

Regarding the force terms, we introduce the functional $F_{m,r}\left[f,g\right]$ defined in \eqref{F'}, 
which possesses the estimate
 \begin{align*}
F_{m,r}\left[f,g\right]\leq 2r^{\frac{2m-5}{m}}F_{m,1}\left[f,g\right],
\end{align*} 
as shown in Lemma \ref{lem3-2} below. From this expression, the condition $m>\frac{5}{2}$ clearly guarantees that 
$F_{m,r}$ admits well-behaved decay, and this feature can address the effects arising from the external force term 
in subsequent iterations.


\begin{lemma}\label{lem3-2}
Let $m>\frac{5}{2}$ and $\lambda_0:=\frac{1}{8^{\frac{m}{2m-5}}}$. Then for any $f,g\in L^m(Q_1)$, 
\begin{align}\label{equ3-23}
F_{m,\lambda_0}\left[f,g\right]\leq \frac{1}{4}F_{m,1}\left[f,g\right].
\end{align}
Moreover, for all $\lambda\in(0,\frac{1}{8^{\frac{m}{2m-5}}}]$ and $i\in \Bbb{N}^+$,
\begin{align}\label{equ3-23'}
F_{m,\lambda^i}\left[f,g\right]\leq \frac{1}{4^i} F_{m,1}\left[f,g\right].
\end{align}
\end{lemma}
\begin{proof}
First, we observe that for any $r\in(0,1)$,
\begin{align*}
r^{2m}\dashint_{Q_{r}}|f|^{m}\,dxdt =\frac{r^{2m-5}}{|B_{1}|}\int_{Q_{r}}|f|^{m}\,dxdt 
\leq r^{2m-5}\dashint_{Q_{1}}|f|^{m}\,dxdt.
\end{align*}
This tells us that for any $r\in(0,1)$,
\begin{align*}
r^{2}\sqrt[m]{\dashint_{Q_{r}}|f|^{m}\,dxdt}
\leq r^{\frac{2m-5}{m}}\sqrt[m]{\dashint_{Q_{1}}|f|^{m}\,dxdt}.
\end{align*}
Similarly, the external force $g$ possesses the same characteristics:
\begin{align*}
r^{2}\sqrt[m]{\dashint_{Q_{r}}|g|^{m}\,dxdt}
\leq r^{\frac{2m-5}{m}}\sqrt[m]{\dashint_{Q_{1}}|g|^{m}\,dxdt},\quad \forall~r\in(0,1).
\end{align*}
From the above two inequalities, it is obvious that
\begin{align*}
F_{m,r}\left[f,g\right]\leq 2r^{\frac{2m-5}{m}}F_{m,1}\left[f,g\right].
\end{align*}
Since $m>\frac{5}{2}$, we may choose $\lambda_0=\frac{1}{8^{\frac{m}{2m-5}}}$  to ensure \eqref{equ3-23}.

To prove \eqref{equ3-23'}, note that for $\lambda\in(0,\lambda_0]$ and $i\in \Bbb{N}^+$,
\begin{align*}
\lambda^{2mi}\dashint_{Q_{\lambda^{i}}}|f|^{m}\,dxdt =\frac{\lambda^{2mi-5i}}{|B_{1}|}\int_{Q_{\lambda^{i}}}|f|^{m}\,dxdt 
\leq \lambda^{2mi-5i}\dashint_{Q_{1}}|f|^{m}\,dxdt,
\end{align*}
which implies
\begin{align*}
\lambda^{2i}\sqrt[m]{\dashint_{Q_{\lambda^{i}}}|f|^{m}\,dxdt}
\leq \lambda^{\frac{2mi-5i}{m}}\sqrt[m]{\dashint_{Q_{1}}|f|^{m}\,dxdt}, \quad \forall~\lambda\in(0,\lambda_0],~ i\in \Bbb{N}^+.
\end{align*}
An analogous estimate holds for the external force $g$:
\begin{align*}
\lambda^{2i}\sqrt[m]{\dashint_{Q_{\lambda^{i}}}|g|^{m}\,dxdt}
\leq \lambda^{\frac{2mi-5i}{m}}\sqrt[m]{\dashint_{Q_{1}}|g|^{m}\,dxdt},\quad \forall~\lambda\in(0,\lambda_0],~ i\in \Bbb{N}^+.
\end{align*}
Combining these two estimates gives
\begin{align*}
F_{m,\lambda^{i}}\left[f,g\right]
\leq 2\lambda^{\frac{2mi-5i}{m}}F_{m,1}\left[f,g\right]
\leq \frac{2}{8^i}F_{m,1}\left[f,g\right]
\leq \frac{1}{4^i}F_{m,1}\left[f,g\right], \quad \forall~ i\in \Bbb{N}^+.
\end{align*}
The proof is finished.
\end{proof}

Now, we utilize the following Lemma (Lemma \ref{lem3-1}) to control the perturbation between $(u-V_1,B-M_1)$ 
and its approximate solution $(v_1,b_1)$, where $v_1$ and $b_1$ solve the corresponding linearized equations.

\begin{lemma} \label{lem3-1}
For any $ \epsilon>0$, we can find a small positive constant $\delta_{0}=\delta_{0}(\epsilon)$ with the property that if 
 $(u_{1},B_{1},P_{1})$ is a suitable weak solution of \eqref{equ3-1} in $Q_{1}$ with some $r_1\le1$,
some vectors $|V_1|\le1$, $|M_1|\le1$, and satisfies
\begin{align*}
C_{1}\left[u_{1}, B_{1}, P_{1}\right]+\sqrt{F_{m,1}\left[f_{1}, g_{1}\right]}\leq \delta_{0},
\end{align*}
then there exist $r\le1$, vectors $|V|\le1, |M|\le1$ and solutions $(v_{1},b_{1},q_{1})$ to the linearized MHD system as well as $H_{1}(x,t)$ of the harmonic flow as
\begin{align}\label{equ3-4}
\left\{
  \begin{array}{ll}
    \partial_{t}v_{1}-\Delta v_{1}+rV\cdot\nabla v_{1}-rM\cdot\nabla b_{1}+\nabla q_{1}=0    &\quad {\rm in}~ Q_{\frac{1}{2}},\\
    \partial_{t}b_{1}-\Delta b_{1}+rV\cdot\nabla b_{1}-rM\cdot\nabla v_{1}=0                 &\quad {\rm in}~ Q_{\frac{1}{2}},\\
    \Delta H_{1}=0                                                                           &\quad {\rm in}~ Q_{\frac{1}{2}},\\
    {\rm div}\,v_{1}=0,\quad {\rm div}\,b_{1}=0                                                  &\quad {\rm in}~ Q_{\frac{1}{2}},
     \end{array}
\right.
\end{align}
such that
\begin{align*}
C_{\frac{1}{2}}\left[v_{1},b_{1},q_{1}\right]\leq& 2^{3}\left(C_{1}\left[u_{1},B_{1},P_{1}\right]+\sqrt{F_{m,1}\left[f_{1},g_{1}\right]}\right), \\
C_{\frac{1}{2}}\left[0,0,H_{1}\right]\leq& 2^{3}\left(C_{1}\left[u_{1},B_{1},P_{1}\right]+\sqrt{F_{m,1}\left[f_{1},g_{1}\right]}\right),
\end{align*}
and
\begin{align*}
C_{\frac{1}{2}}\left[u_{1}-v_{1},B_{1}-b_{1},P_{1}-H_{1}\right]
\leq \epsilon \left(C_{1}\left[u_{1},B_{1},P_{1}\right]+\sqrt{F_{m,1}\left[f_{1},g_{1}\right]}\right).
\end{align*}
\end{lemma}
\begin{proof}
We prove this lemma by contradiction arguments. Suppose it were not true, there would exist an $\epsilon_{0}>0$ such that 
for any $n\in \Bbb{N}^+$, one could find $u_{1n}$, $B_{1n}$, $P_{1n}$, $f_{1n}$, $g_{1n}$ satisfying
 \eqref{equ3-1} with $r_{1n}\le1$, $|V_{1n}|\le1,|M_{1n}|\le1$ and 
\begin{align}\label{equ3-5'}
C_{1}\left[u_{1n},B_{1n},P_{1n}\right]+\sqrt{F_{m,1}\left[f_{1n},g_{1n}\right]}\leq \frac{1}{n},
\end{align}
but for any $r\le1$ and $|V|\le1,|M|\le1$, the corresponding solutions $(v_{1},b_{1},q_{1},H_{1})$ of \eqref{equ3-4} with the properties
\begin{align}
C_{\frac{1}{2}}\left[v_{1},b_{1},q_{1}\right]&\leq 2^{3}\left(C_{1}\left[u_{1n},B_{1n},P_{1n}\right]+\sqrt{F_{m,1}\left[f_{1n},g_{1n}\right]}\right),\label{add1}\\
C_{\frac{1}{2}}\left[0,0,H_{1}\right]&\leq 2^{3}\left(C_{1}[u_{1n},B_{1n},P_{1n}]+\sqrt{F_{m,1}\left[f_{1n},g_{1n}\right]}\right)\label{add2},
\end{align}
would satisfy
\begin{align*}
C_{\frac{1}{2}}\left[u_{1n}-v_{1},B_{1n}-b_{1},P_{1n}-H_{1}\right]> \epsilon_{0} \left(C_{1}\left[u_{1n},B_{1n},P_{1n}\right]+\sqrt{F_{m,1}\left[f_{1n},g_{1n}\right]}\right).
\end{align*}

Let us first abbreviate 
$$C_{1n}:=C_{1}\left[u_{1n},B_{1n},P_{1n}\right],\quad F_{m,1n}:=F_{m,1}\left[f_{1n},g_{1n}\right],$$ 
and denote
\begin{align*}
\tilde{u}_{1n}=\frac{u_{1n}}{C_{1n}+\sqrt{F_{m,1n}}},\qquad\tilde{B}_{1n}=\frac{B_{1n}}{C_{1n}+\sqrt{F_{m,1n}} },
\qquad\tilde{P}_{1n}=\frac{P_{1n}}{C_{1n}+\sqrt{F_{m,1n}}}
\end{align*}
as well as
\begin{align*}
\tilde{f}_{1n}=\frac{f_{1n}}{C_{1n}+\sqrt{F_{m,1n}}},\qquad \tilde{g}_{1n}=\frac{g_{1n}}{C_{1n}+\sqrt{F_{m,1n}}}.
\end{align*}
Then, the definitions of $\tilde{u}_{1n}$, $\tilde{B}_{1n}$, $\tilde{P}_{1n}$, $\tilde{f}_{1n}$ and $\tilde{g}_{1n}$ imply that
\begin{align}\label{danweihua}
 C_{1}\left[\tilde{u}_{1n},\tilde{B}_{1n},\tilde{P}_{1n}\right]+\sqrt{F_{m,1}\left[\tilde{f}_{1n},\tilde{g}_{1n}\right]}\le 1
\end{align}
and
\begin{align}\label{danweihua'}
F_{m,1}\left[\tilde{f}_{1n},\tilde{g}_{1n}\right]\le \frac{1}{n}.
\end{align}
By some calculations, we find that $(\tilde{u}_{1n},\tilde{B}_{1n},\tilde{P}_{1n})$ satisfies
\begin{align}\label{equ3-6}
\left\{
  \begin{array}{ll}
    \partial_{t}\tilde{u}_{1n}-\Delta \tilde{u}_{1n}
    +\left(r_{1n}V_{1n}+r_{1n}\left(C_{1n}+\sqrt{F_{m,1n}}\right)\tilde{u}_{1n}\right)\cdot\nabla \tilde{u}_{1n}\\[2mm]
    \quad\quad\quad\quad\quad\quad\,
    -\left(r_{1n}M_{1n}+r_{1n}\left(C_{1n}+\sqrt{F_{m,1n}}\right)\tilde{B}_{1n}\right)\cdot\nabla \tilde{B}_{1n}+\nabla\tilde{P}_{1n}=\tilde{f}_{1n} &\quad {\rm in}~Q_{1},\\[2mm]
    \partial_{t}\tilde{B}_{1n}-\Delta \tilde{B}_{1n}
    +\left(r_{1n}V_{1n}+r_{1n}\left(C_{1n}+\sqrt{F_{m,1n}}\right)\tilde{u}_{1n}\right)\cdot\nabla \tilde{B}_{1n}\\[2mm]
    \quad\quad\quad\quad\quad\quad\,\,\,
    -\left(r_{1n}M_{1n}+r_{1n}\left(C_{1n}+\sqrt{F_{m,1n}}\right)\tilde{B}_{1n}\right)\cdot\nabla \tilde{u}_{1n}=\tilde{g}_{1n}                      &\quad  {\rm in}~Q_{1},\\[2mm]
    {\rm div}\,\tilde{u}_{1n}=0,\quad {\rm div}\,\tilde{B}_{1n}=0                     
     &\quad {\rm in}~ Q_{1},
     \end{array}
\right.
\end{align}
and 
\begin{align}\label{contradiction}
C_{\frac{1}{2}}\left[\tilde{u}_{1n}-\frac{v_{1}}{C_{1n}+\sqrt{F_{m,1n}}},
\tilde{B}_{1n}-\frac{b_{1}}{C_{1n}+\sqrt{F_{m,1n}}},
\tilde{P}_{1n}-\frac{H_{1}}{C_{1n}+\sqrt{F_{m,1n}}}\right]
> \epsilon_{0}
\end{align}
for any solutions $(v_{1},b_{1},q_{1},H_{1})$ of \eqref{equ3-4} with $r\le1$ and $|V|\le1,|M|\le1$ admitting properties \eqref{add1} and \eqref{add2}.

By selecting a proper testing function for \eqref{equ3-6}
and utilizing the H\"{o}lder inequality as well as the Young inequality, one can readily arrive at the following uniform-in-$n$ estimate
\begin{align}\label{energy}
\sup_{t\in(-\frac{9}{16},0)}\int_{B_{\frac{3}{4}}}\left(|\tilde{u}_{1n}(x,t)|^{2}+|\tilde{B}_{1n}(x,t)|^{2}\right)dx
+\int_{Q_{\frac{3}{4}}}\left(|\nabla \tilde{u}_{1n}|^{2}+|\nabla \tilde{B}_{1n}|^{2}\right)dxds
\leq C.
\end{align}
The detailed proof for \eqref{energy} is omitted. Thus, it holds that
\begin{align}\label{equ3-9}
\|(\tilde{u}_{1n}, \tilde{B}_{1n})\|_{L^{\infty}\left((-9/16,0);L^{2}(B_{3/4})\right)
\cap L^{2}\left((-9/16,0);L^{6}(B_{3/4})\right)}\le C.
\end{align}
Based on this, we invoke the Lebesgue interpolation (see e.g., in \cite[Theorem 1.5]{RRS2016})  
to state that 
\begin{align}\label{equ3-10}
\{\tilde{u}_{1n}\}, \{\tilde{B}_{1n}\}\text{ are bounded in } L^{\frac{10}{3}}(Q_{\frac{3}{4}}).
\end{align}


Let $\varphi(x)\in C^{\infty}_{0}(B_{\frac{3}{4}})$ be a spatial variable function with ${\rm div}\,\varphi=0$.
Multiplying \eqref{equ3-6}$_1$ by $\varphi$ and integrating over $B_{\frac{3}{4}}$, one can see for a.e. $t>0$
\begin{align}\label{equ3-11}
\left\langle \partial_{t}\tilde{u}_{1n},\varphi \right\rangle_{W^{-1,3}(B_{\frac{3}{4}})\times W^{1,3/2}_{0}(B_{\frac{3}{4}}) }
=&-\left\langle \nabla\tilde{u}_{1n},\nabla\varphi \right\rangle
-\left\langle r_{1n}V_{1n}\cdot\nabla\tilde{u}_{1n},\varphi \right\rangle \nonumber\\
&-\left\langle r_{1n}\left(C_{1n}+\sqrt{F_{m,1n}}\right)\tilde{u}_{1n}\cdot\nabla\tilde{u}_{1n},\varphi\right \rangle
+\left\langle r_{1n}M_{1n}\cdot\nabla\tilde{B}_{1n},\varphi\right\rangle \nonumber\\
&+\left\langle r_{1n}\left(C_{1n}+\sqrt{F_{m,1n}}\right)\tilde{B}_{1n}\cdot\nabla\tilde{B}_{1n},\varphi \right\rangle
-\left\langle \nabla\tilde{P}_{1n},\varphi \right\rangle
+\left\langle \tilde{f}_{1n},\varphi \right\rangle. 
\end{align}
In the above display, the term $\left\langle \nabla\tilde{P}_{1n},\varphi \right\rangle$ vanishes because $\varphi$ is divergence free.
It follows from the H\"{o}lder inequality and the Sobolev embedding theorem that for a.e. $t\in(-\frac{9}{16},0)$,
\begin{align*}
\left|\left\langle \nabla\tilde{u}_{1n},\nabla\varphi \right\rangle\right|
\leq& \left\|\nabla\tilde{u}_{1n}\right\|_{L^2(B_{\frac{3}{4}})}\left\|\nabla\varphi\right\|_{L^2(B_{\frac{3}{4}})},\\
\left|\left\langle r_{1n}V_{1n}\cdot\nabla\tilde{u}_{1n},\varphi \right\rangle\right|
\leq& \left\|\nabla\tilde{u}_{1n}\right\|_{L^2(B_{\frac{3}{4}})}\left\|\varphi\right\|_{L^2(B_{\frac{3}{4}})},\\
\left|\left\langle r_{1n}\left(C_{1n}+\sqrt{F_{m,1n}}\right)\tilde{u}_{1n}\cdot\nabla\tilde{u}_{1n},\varphi \right\rangle\right|
\leq& \left\|\tilde{u}_{1n}\right\|_{L^{3}(B_{\frac{3}{4}})}\left\|\nabla\tilde{u}_{1n}\right\|_{L^{2}(B_{\frac{3}{4}})}\left\|\varphi\right\|_{L^6(B_{\frac{3}{4}})}\\
\leq& \left\|\tilde{u}_{1n}\right\|_{L^{3}(B_{\frac{3}{4}})}\left\|\nabla\tilde{u}_{1n}\right\|_{L^{2}(B_{\frac{3}{4}})}\left\|\varphi\right\|_{H^1(B_{\frac{3}{4}})}.
\end{align*}
By the similar reasonings, we also have that for a.e. $t\in(-\frac{9}{16},0)$,
\begin{align*}
\left|\left\langle r_{1n}M_{1n}\cdot\nabla\tilde{B}_{1n},\varphi \right\rangle\right|
\leq& \left\|\nabla\tilde{B}_{1n}\right\|_{L^2(B_{\frac{3}{4}})}\left\|\varphi\right\|_{L^2(B_{\frac{3}{4}})},\\
\left|\left\langle r_{1n}\left(C_{1n}+\sqrt{F_{m,1n}}\right)\tilde{B}_{1n}\cdot\nabla\tilde{B}_{1n},\varphi \right\rangle\right|
\leq& \left\|\tilde{B}_{1n}\right\|_{L^{3}(B_{\frac{3}{4}})}\left\|\nabla\tilde{B}_{1n}\right\|_{L^{2}(B_{\frac{3}{4}})}\left\|\varphi\right\|_{H^1(B_{\frac{3}{4}})},\\
\left|\left\langle\tilde{f}_{1n},\varphi \right\rangle\right|
\leq&\left\|\tilde{f}_{1n}\right\|_{L^{m}(B_{\frac{3}{4}})}\left\|\varphi\right\|_{L^{\frac{m}{m-1}}(B_{\frac{3}{4}})}.
\end{align*}
Inserting the above estimates into \eqref{equ3-11}, 
we then employ the H\"{o}lder inequality to deduce 
\begin{align*}
&\int_{-\frac{9}{16}}^{0}\left\|\partial_{t}\tilde{u}_{1n}\right\|^{\frac{3}{2}}_{W^{-1,3}(B_{\frac{3}{4}})}\,dt\\
\leq& C\int_{-\frac{9}{16}}^{0}\left(\left\|\nabla\tilde{u}_{1n}\right\|^{\frac{3}{2}}_{L^{2}(B_{\frac{3}{4}})}
+\left\|\nabla\tilde{B}_{1n}\right\|^{\frac{3}{2}}_{L^{2}(B_{\frac{3}{4}})}
+\left\|\tilde{f}_{1n}\right\|^{\frac{3}{2}}_{L^{m}(B_{\frac{3}{4}})}\right)\,dt\\
\leq& C\left(\int_{-\frac{9}{16}}^{0}\left\|\nabla\tilde{u}_{1n}\right\|^{2}_{L^{2}(B_{\frac{3}{4}})}\,dt\right)^{\frac{3}{4}}
+C\left(\int_{-\frac{9}{16}}^{0}\left\|\nabla\tilde{B}_{1n}\right\|^{2}_{L^{2}(B_{\frac{3}{4}})}\,dt\right)^{\frac{3}{4}}
+C\left(\int_{-\frac{9}{16}}^{0}\left\|\tilde{f}_{1n}\right\|^{m}_{L^{m}(B_{\frac{3}{4}})}\,dt\right)^{\frac{3}{2m}},
\end{align*}
which combined with \eqref{energy} indicates that 
\begin{align*}
\int_{-\frac{9}{16}}^{0}\left\|\partial_{t}\tilde{u}_{1n}\right\|^{\frac{3}{2}}_{W^{-1,3}(B_{\frac{3}{4}})}\,dt\le C.
\end{align*}
This along with \eqref{energy} and \eqref{equ3-10} allows us to apply the Aubin-Lions lemma and 
derive that 
\begin{align}\label{equ3-13}
\tilde{u}_{1n}\rightarrow \tilde{u}_{1\infty}\quad{\rm in}~ L^{3}(Q_{\frac{3}{4}}).
\end{align}
Performing a similar argument for $\{\tilde{B}_{1n}\}$, we conclude that
\begin{align}\label{equ3-14}
\tilde{B}_{1n}\rightarrow \tilde{B}_{1\infty}\quad{\rm in}~ L^{3}(Q_{\frac{3}{4}}).
\end{align}

Since $\tilde{P}_{1n}$ satisfies the following equation
\begin{align*}
\Delta\tilde{P}_{1n}=-r_{1n}\left(C_{1n}+\sqrt{F_{m,1n}}\right)\sum_{i,j=1}^{3}\frac{\partial^{2}}{\partial x_{i}\partial x_{j}}
\left(\tilde{u}^{i}_{1n}\tilde{u}^{j}_{1n}-\tilde{B}^{i}_{1n}\tilde{B}^{j}_{1n}\right).
\end{align*}
We consider the Dirichlet boundary value problem
\begin{align*}
\left\{
  \begin{array}{ll}
  \Delta \tilde{R}_{1n}=-r_{1n}\left(C_{1n}+\sqrt{F_{m,1n}}\right)\sum_{i,j=1}^{3}\frac{\partial^{2}}{\partial x_{i}\partial x_{j}}
\left(\tilde{u}^{i}_{1n}\tilde{u}^{j}_{1n}-\tilde{B}^{i}_{1n}\tilde{B}^{j}_{1n}\right)  &\quad {\rm in}~ Q_{\frac{3}{4}},\\
\,\,\,\,\, \tilde{R}_{1n}=0                               &\quad {\rm in}~ \partial B_{\frac{3}{4}}\times (-\frac{9}{16},0),
     \end{array}
\right.
\end{align*}
and invoke the Calder\'{o}n-Zygmund estimates to infer that for a.e. $t\in(-\frac{9}{16},0)$,
\begin{align*}
\int_{B_{\frac{3}{4}}}|\tilde{R}_{1n}|^{\frac{3}{2}}\,dx
\leq Cr_{1n}\left(C_{1n}+\sqrt{F_{m,1n}}\right)\int_{B_{\frac{3}{4}}}\left(|\tilde{u}_{1n}|^{3}+|\tilde{B}_{1n}|^{3}\right)\,dx.
\end{align*}
Integrating over $(-\frac{9}{16},0)$ w.r.t. $t$ and combining with \eqref{equ3-5'} and \eqref{equ3-10}, we conclude
\begin{align}\label{equ3-14'}
\tilde{R}_{1n}\rightarrow 0\quad{\rm in}~ L^{\frac{3}{2}}(Q_{\frac{3}{4}}).
\end{align}
We can choose a sufficiently large $\bar{n} \in \Bbb{N}$ for which $Cr_{1\bar{n}}\left(C_{1\bar{n}}+\sqrt{F_{m,1\bar{n}}}\right)\leq 1$.
Thus, there holds that 
\begin{align}\label{equ3-14''}
\dashint_{Q_{\frac{1}{2}}}|\tilde{R}_{1\bar{n}}|^{\frac{3}{2}}\,dx
\leq \dashint_{Q_{\frac{1}{2}}}\left(|\tilde{u}_{1\bar{n}}|^{3}+|\tilde{B}_{1\bar{n}}|^{3}\right)\,dx,
\end{align}
whenever $n>\bar{n}$.

Moreover, from \eqref{danweihua} and \eqref{danweihua'}, it follows that
\begin{align*}
\tilde{P}_{1n}\rightharpoonup \tilde{P}_{1\infty}\quad {\rm in}~ L^{\frac{3}{2}}(Q_{\frac{3}{4}}),\qquad
\tilde{f}_{1n}\rightarrow 0,\quad 
\tilde{g}_{1n}\rightarrow 0 \quad{\rm in}~ L^{m}(Q_{\frac{3}{4}}).
\end{align*}
Due to $r_{1n}\le1$, $|V_{1n}|\le1,|M_{1n}|\le1$, we have
\begin{align*}
r_{1n}\rightarrow r_{1\infty}\quad{\rm in}~ \mathbb{R}^1~~~
{\rm and}~~~
V_{1n}\rightarrow V_{1\infty},\quad 
M_{1n}\rightarrow M_{1\infty} \quad{\rm in}~ \mathbb{R}^3.
\end{align*}
Hence, $(\tilde{u}_{1\infty},\tilde{B}_{1\infty}, \tilde{P}_{1\infty})$ solves the following equations in the weak sense:
\begin{align}\label{equ3-8''}
\left\{
  \begin{array}{ll}
    \partial_{t}\tilde{u}_{1\infty}-\Delta \tilde{u}_{1\infty}
    +r_{1\infty}V_{1\infty}\cdot\nabla \tilde{u}_{1\infty}-r_{1\infty}M_{1\infty}\cdot\nabla \tilde{B}_{1\infty}
    +\nabla \tilde{P}_{1\infty}=0                                                               &\quad {\rm in}~ Q_{\frac{3}{4}},\\
    \partial_{t}\tilde{B}_{1\infty}-\Delta \tilde{B}_{1\infty}
    +r_{1\infty}V_{1\infty}\cdot\nabla \tilde{B}_{1\infty}-r_{1\infty}M_{1\infty}\cdot\nabla \tilde{u}_{1\infty}=0  &\quad {\rm in}~ Q_{\frac{3}{4}},\\
    {\rm div}\,\tilde{u}_{1\infty}=0,\quad {\rm div}\,\tilde{B}_{1\infty}=0                         &\quad {\rm in}~ Q_{\frac{3}{4}}.
     \end{array}
\right.
\end{align}

Then, one can tell
\begin{align*}
    &C_{\frac{1}{2}}\left[\tilde{u}_{1n}-\tilde{u}_{1\infty}, \tilde{B}_{1n}-\tilde{B}_{1\infty}, \tilde{P}_{1n}-(\tilde{P}_{1n}-\tilde{R}_{1n})\right]\\
    =&\sqrt[3]{\frac{1}{|Q_{\frac{1}{2}}|}\int_{Q_{\frac{1}{2}}}|\tilde{u}_{1n}-\tilde{u}_{1\infty}|^{3}}
    +\sqrt[3]{\frac{1}{|Q_{\frac{1}{2}}|}\int_{Q_{\frac{1}{2}}}|\tilde{B}_{1n}-\tilde{B}_{1\infty}|^{3}}
    +\frac{1}{2}\sqrt[\frac{3}{2}]{\frac{1}{|Q_{\frac{1}{2}}|}\int_{Q_{\frac{1}{2}}}|\tilde{R}_{1n}|^{\frac{3}{2}}}\\
     \leq&2
     \left(\sqrt[3]{\frac{1}{|Q_{\frac{3}{4}}|}\int_{Q_{\frac{3}{4}}}|\tilde{u}_{1n}-\tilde{u}_{1\infty}|^{3}}
    +\sqrt[3]{\frac{1}{|Q_{\frac{3}{4}}|}\int_{Q_{\frac{3}{4}}}|\tilde{B}_{1n}-\tilde{B}_{1\infty}|^{3}}
    +\sqrt[\frac{3}{2}]{\frac{1}{|Q_{\frac{3}{4}}|}\int_{Q_{\frac{3}{4}}}|\tilde{R}_{1n}|^{\frac{3}{2}}}\right),
\end{align*}
where the right-hand side converges to 0 as $n\rightarrow0$ because of \eqref{equ3-13}--\eqref{equ3-14'}.
Thus, there exists $\hat{n}\in \Bbb{N}^+$ $(\hat{n}>\bar{n})$ sufficiently large such that 
\begin{align}\label{contradiction'}
    C_{\frac{1}{2}}\left[\tilde{u}_{1\hat{n}}-\tilde{u}_{1\infty}, 
    \tilde{B}_{1\hat{n}}-\tilde{B}_{1\infty}, 
    \tilde{P}_{1\hat{n}}-\left(\tilde{P}_{1\hat{n}}-\tilde{R}_{1\hat{n}}\right)\right]
    \le \frac{\epsilon_0}{2}.
\end{align}
Now, we define 
\begin{align*}
   v_{1\infty}&=\left(C_{1\hat{n}}+\sqrt{F_{m,1\hat{n}}}\right)\tilde{u}_{1\infty},  ~~\, b_{1\infty}=\left(C_{1\hat{n}}+\sqrt{F_{m,1\hat{n}}}\right)\tilde{B}_{1\infty}, \\
   q_{1\infty}&=\left(C_{1\hat{n}}+\sqrt{F_{m,1\hat{n}}}\right)\tilde{P}_{1\infty}, ~~
   H_{1\hat{n}}=\left(C_{1\hat{n}}+\sqrt{F_{m,1\hat{n}}}\right)\left(\tilde{P}_{1\hat{n}}-\tilde{R}_{1\hat{n}}\right).
\end{align*}
Substituting $v_{1\infty}$, $b_{1\infty}$ and $H_{1\hat{n}}$ into \eqref{contradiction'}, we have
\begin{align}\label{contradiction''}
    &C_{\frac{1}{2}}\left[\tilde{u}_{1\hat{n}}-\frac{ v_{1\infty}}{      C_{1\hat{n}}+\sqrt{F_{m,1\hat{n}}}    },
     \tilde{B}_{1\hat{n}}-\frac{b_{1\infty}}{  C_{1\hat{n}}+\sqrt{F_{m,1\hat{n}}}  }, 
     \tilde{P}_{1\hat{n}}-\frac{H_{1\hat{n}}}{   C_{1\hat{n}}+\sqrt{F_{m,1\hat{n}}} }\right]
     \le \frac{\epsilon_0}{2}.
\end{align}
We also deduce that $(v_{1\infty},b_{1\infty}, q_{1\infty})$ satisfies 
\begin{align*}
C_{\frac{1}{2}}\left[v_{1\infty},b_{1\infty}, q_{1\infty}\right]
&\le \left(C_{1\hat{n}}+\sqrt{F_{m,1\hat{n}}}\right)\,
C_{\frac{1}{2}}\left[\tilde{u}_{1\infty},\tilde{B}_{1\infty},\tilde{P}_{1\infty}\right]\\
&\le 2^3 \left(C_{1\hat{n}}+\sqrt{F_{m,1\hat{n}}}\right)\,C_{1}\left[\tilde{u}_{1\infty},\tilde{B}_{1\infty},\tilde{P}_{1\infty}\right].
\end{align*}
By combining the lower semicontinuity of the functional $C_{1}[\cdot,\cdot,\cdot]$ and \eqref{danweihua}, it holds
\begin{align*}
C_{1}\left[\tilde{u}_{1\infty},\tilde{B}_{1\infty},\tilde{P}_{1\infty}\right]
\leq\liminf_{n\rightarrow\infty} C_{1}\left[\tilde{u}_{1n},\tilde{B}_{1n},\tilde{P}_{1n}\right]
\le 1.
\end{align*}
Thus,
\begin{align}\label{add1'}
C_{\frac{1}{2}}\left[v_{1\infty},b_{1\infty}, q_{1\infty}\right]
\leq2^3 \left(C_{1\hat{n}}+\sqrt{F_{m,1\hat{n}}}\right).
\end{align}
Besides, from the definition of $H_{1\hat{n}}$ and \eqref{equ3-14''}, we obtain
\begin{align*}
C_{\frac{1}{2}}\left[0,0, H_{1\hat{n}}\right]
\leq&\frac{\left(C_{1\hat{n}}+\sqrt{F_{m,1\hat{n}}}\right)}{2}\sqrt[\frac{3}{2}]
{\dashint_{Q_{\frac{1}{2}}}\left(|\tilde{u}_{1\hat{n}}|^{3}+|\tilde{B}_{1\hat{n}}|^{3}\right)\,dxdt}
+\frac{\left(C_{1\hat{n}}+\sqrt{F_{m,1\hat{n}}}\right)}{2}
\sqrt[\frac{3}{2}]{\dashint_{Q_{\frac{1}{2}}}|\tilde{P}_{1\hat{n}}|^{\frac{3}{2}}\,dxdt}\\
\le& 2^3 \left(C_{1\hat{n}}+\sqrt{F_{m,1\hat{n}}}\right)
\left(\sqrt[\frac{3}{2}]{{\dashint_{Q_{1}}\left(|\tilde{u}_{1\hat{n}}|^{3}+|\tilde{B}_{1\hat{n}}|^{3}\right)\,dxdt}}
+\sqrt[\frac{3}{2}]{\dashint_{Q_{1}}|\tilde{P}_{1\hat{n}}|^{\frac{3}{2}}\,dxdt}\right).
\end{align*}
In view of \eqref{danweihua}, this simplifies to 
\begin{align}\label{add2'}
C_{\frac{1}{2}}\left[0,0, H_{1\hat{n}}\right]
\le 2^3 \left(C_{1\hat{n}}+\sqrt{F_{m,1\hat{n}}}\right).
\end{align}
Therefore, we have constructed a solution $(v_{1},b_{1},q_{1},H_{1})=(v_{1\infty},b_{1\infty},q_{1\infty},H_{1\hat{n}})$ 
to \eqref{equ3-4} with $r=r_{1\infty}$ and $V=V_{1\infty}$, $M=M_{1\infty}$,
satisfying the uniform bounds $r_{1\infty}\le1$, $|V_{1\infty}|\le1,|M_{1\infty}|\le1$,
together with the key estimates \eqref{contradiction''}--\eqref{add2'}.
However, inequality \eqref{contradiction''} directly contradicts \eqref{contradiction}.
This completes the proof.
\end{proof}


\begin{proposition}\label{pro3-1}
Let $\lambda_0\in(0,1)$ be as in Lemma \ref{lem3-2}.
There are universal constants $\lambda_1\in(0,\lambda_0]$, $\delta_1\in(0,1)$ and $\Lambda>0$ such that 
if $(u_{1},B_{1},P_{1})$ is a suitable weak solution of \eqref{equ3-1} 
in $Q_{1}$ with some $r_1\le1$ and $|V_1|\le1,|M_1|\le1$ satisfying   
\begin{align*}
C_{1}\left[u_{1},B_{1},P_{1}\right]+\sqrt{F_{m,1}\left[f_{1},g_{1}\right]}\leq \delta_{1},
\end{align*}
then we can find constant vectors $\overline{V}_{2}$, $\overline{M}_{2}\in \mathbb{R}^3$ and a function of time $\overline{h}_{2}(t)$ such that
\begin{align*}
|\overline{V}_{2}|\leq& \frac{\Lambda}{2}C_{1}\left[u_{1},B_{1},P_{1}\right]+\frac{\Lambda}{2}\sqrt{F_{m,1}\left[f_{1},g_{1}\right]},\\
|\overline{M}_{2}|\leq& \frac{\Lambda}{2}C_{1}\left[u_{1},B_{1},P_{1}\right]+\frac{\Lambda}{2}\sqrt{F_{m,1}\left[f_{1},g_{1}\right]} 
\end{align*}
and 
\begin{align*}
C_{\lambda_1}\left[u_{1}-\overline{V}_{2},B_{1}-\overline{M}_{2},P_{1}-\overline{h}_{2}(t)\right]
\leq \frac{1}{2}C_{1}\left[u_{1},B_{1},P_{1}\right]+\frac{1}{2}\sqrt{F_{m,1}\left[f_{1},g_{1}\right]}.
\end{align*}
\end{proposition}
\begin{proof}
Given $\epsilon\in(0,1)$ to be determined later, we accordingly take $\delta_0=\delta_{0}(\epsilon)>0$ by applying Lemma \ref{lem3-1}.
Setting $\delta:=C_{1}\left[u_{1},B_{1},P_{1}\right]+\sqrt{F_{m,1}\left[f_{1},g_{1}\right]}$ and assuming $\delta\le\delta_0 $,
we define $v_{1},b_{1}, q_{1}, H_{1}$ as in Lemma \ref{lem3-1} with 
$C_{\frac{1}{2}}\left[v_{1},b_{1},q_{1}\right]\leq 2^{3}\delta$, $C_{\frac{1}{2}}\left[0,0,H_{1}\right]\leq 2^{3}\delta$
and 
\begin{align}\label{equ3-55'}
C_{\frac{1}{2}}\left[u_{1}-v_{1},B_{1}-b_{1},P_{1}-H_{1}\right]
\leq \epsilon \delta.
\end{align}
We know that $v_{1}, b_{1}$ are locally smooth in space and $C^{\frac{1}{3}}$ in time variable. Consequently, 
for each $(x,t)\in Q_{r}$ with $0<r\leq\frac{1}{2}$, we get
\begin{align}\label{equ3-15}
|v_{1}(x,t)-v_{1}(0,0)|^{3}\leq Cr\delta^{3},\qquad
|b_{1}(x,t)-b_{1}(0,0)|^{3}\leq Cr\delta^{3}
\end{align}
and 
\begin{align}\label{equ3-15'}
|v_{1}(0,0)|\leq C\delta,\qquad
|b_{1}(0,0)|\leq C\delta.
\end{align}
Additionally, applying the gradient estimates to the harmonic function $H_{1}$, we have that for each $(x,t)\in Q_{r}$ $(0<r\leq\frac{1}{2})$,
\begin{align}\label{equ3-16}
|H_{1}(x,t)-H_{1}(0,t)|^{\frac{3}{2}}\leq Cr^{\frac{3}{2}}\left\|H_{1}(\cdot,t)\right\|^{\frac{3}{2}}_{L^{\frac{3}{2}}(B_{\frac{1}{2}})}.
\end{align}

Now, we choose 
\begin{align}\label{equ3-17}
\overline{V}_{2}=v_{1}(0,0),\quad\overline{M}_{2}=b_{1}(0,0)\quad{\rm and}\quad \overline{h}_{2}(t)=H_{1}(0,t).
\end{align}
Upon integrating the inequalities in \eqref{equ3-15} over $Q_{r}$, one can obtain that
\begin{align}\label{equ3-17'}
\dashint_{Q_{r}}|v_{1}-\overline{V}_{2}|^{3}\,dxdt\leq Cr\delta^{3},\quad 
\dashint_{Q_{r}}|b_{1}-\overline{M}_{2}|^{3}\,dxdt\leq Cr\delta^{3}.
\end{align}
By integrating \eqref{equ3-16} over $Q_{r}$, together with the fact that $C_{\frac{1}{2}}[0,0,H_{1}]\leq 2^{3}\delta$, we arrive at
\begin{align}\label{equ3-17''}
r^{\frac{3}{2}}\dashint_{Q_{r}}|H_{1}(x,t)-\overline{h}_{2}(t)|^{\frac{3}{2}}\,dxdt
\leq \frac{Cr^{3}}{r^{5}}\int_{-r^{2}}^{0}\left(\int_{B_{r}}\left\|H_{1}\right\|^{\frac{3}{2}}_{L^{\frac{3}{2}}(B_{\frac{1}{2}})}\,dx\right)\,dt
\leq Cr\dashint_{Q_{\frac{1}{2}}}|H_{1}|^{\frac{3}{2}}\,dxdt
\leq Cr\delta^{\frac{3}{2}}.
\end{align}
Combining \eqref{equ3-17'} with \eqref{equ3-55'} gives
\begin{align}\label{equ3-20}
\sqrt[3]{\dashint_{Q_{r}}|u_{1}-\overline{V}_{2}|^{3}\,dxdt}
\leq& \sqrt[3]{\dashint_{Q_{r}}|u_{1}-v_{1}|^{3}\,dxdt}+\sqrt[3]{\dashint_{Q_{r}}|v_{1}-\overline{V}_{2}|^{3}\,dxdt} \nonumber\\
\leq&\frac{1}{(2r)^{\frac{5}{3}}}\sqrt[3]{\dashint_{Q_{\frac{1}{2}}}|u_{1}-v_{1}|^{3}\,dxdt}
+\sqrt[3]{\dashint_{Q_{r}}|v_{1}-\overline{V}_{2}|^{3}\,dxdt}\nonumber\\
\leq&\frac{\epsilon\delta}{(2r)^{\frac{5}{3}}}
+Cr^{\frac{1}{3}}\delta.
\end{align}
Similarly, in light of \eqref{equ3-17'}, \eqref{equ3-17''} and \eqref{equ3-55'}, there holds that
\begin{align}\label{equ3-21}
\sqrt[3]{\dashint_{Q_{r}}|B_{1}-\overline{M}_{2}|^{3}\,dxdt}
\leq\frac{\epsilon \delta}{(2r)^{\frac{5}{3}}}+Cr^{\frac{1}{3}}\delta,
\qquad
 r\sqrt[\frac{3}{2}]{\dashint_{Q_{r}}|P_{1}-\overline{h}_{2}(t)|^{\frac{3}{2}}\,dxdt}
\leq\frac{\epsilon\delta }{(2r)^{\frac{7}{3}}} +Cr^{\frac{2}{3}}\delta.
\end{align}
We select $r_0>0$ sufficiently small to satisfy $2Cr_0^{\frac{1}{3}}+Cr_0^{\frac{2}{3}}\leq \frac{1}{4}$, 
and set $\lambda_1=\min\{r_0,\lambda_0\}$ in the statement of this proposition.
Then we choose $\epsilon_0$ small enough such that $\frac{2\epsilon_0}{(2\lambda_1)^{\frac{5}{3}}}+\frac{\epsilon_0}{(2\lambda_1)^{\frac{7}{3}}}\leq \frac{1}{4}$. 
With these choices of $\lambda_1$ and $\epsilon_0$, we take $\delta_1=\delta_0(\epsilon_0)>0$ as given in Lemma \ref{lem3-1},
 and thereby conclude from \eqref{equ3-20} and \eqref{equ3-21} that
\begin{align*}
C_{\lambda_{1}}\left[u_{1}-\overline{V}_{2},B_{1}-\overline{M}_{2},P_{1}-\overline{h}_{2}(t)\right]
\leq\frac{\delta}{2}.
\end{align*}
Combining this with the definition of $\delta$, together with \eqref{equ3-17} and \eqref{equ3-15'}, we complete the proof.
\end{proof}

\begin{proof}
[\bf Proof of Theorem \ref{the1-1}.]
Assume that $(u,B,P)$ is a suitable weak solution of \eqref{equ1-1}.
We accordingly find a universal constant $\delta_1>0$ from Proposition \ref{pro3-1}.
Then, the smallness condition \eqref{assump1} holds for some $\delta_*\le \delta_1$, and consequently,
\begin{align}\label{add3}
C_{1}\left[u,B,P\right]+\sqrt{F_{m,1}\left[f,g\right]}\leq \delta_{1}.
 \end{align}
Then, the application of Proposition \ref{pro3-1} yields a universal constant $\lambda_1$ such that
\begin{align}\label{equ3-24}
C_{\lambda_1}\left[u-V_{1},B-M_{1},P -h_{1}(t)\right]
\leq \frac{1}{2}C_{1}\left[u,B,P\right]+\frac{1}{2}\sqrt{F_{m,1}\left[f,g\right]},
\end{align}
where the constant vectors $V_1$ and $M_1$ obey
\begin{align}\label{equ3-24''}
|V_{1}|\leq \frac{\Lambda}{2}C_{1}\left[u,B,P\right]+\frac{\Lambda}{2}\sqrt{F_{m,1}\left[f,g\right]},\quad
|M_{1}|\leq \frac{\Lambda}{2}C_{1}\left[u,B,P\right]+\frac{\Lambda}{2}\sqrt{F_{m,1}\left[f,g\right]}.
\end{align}
Let the scaling parameter $r=\lambda_1$ in our coming iterative steps. 
From the scaling definitions \eqref{scaled} and \eqref{scaled'} of $u_{1}$, $B_{1}$, $P_{1}$, $f_{1}$, $g_{1}$,
together with \eqref{equ3-24}, \eqref{equ3-23} and \eqref{assump1}, we derive
\begin{align}\label{equ C_1+F_1}
C_{1}\left[u_{1},B_{1},P_{1}\right]+\sqrt{F_{m,1}\left[f_{1},g_{1}\right]}
&=C_{r}\left[u-V_{1},B-M_{1},P -h_{1}(t)\right]+\sqrt{F_{m,r}\left[f,g\right]}\nonumber\\
&\leq\frac{1}{2}C_{1}\left[u,B,P\right]+\sqrt{F_{m,1}\left[f,g\right]}\\
&\leq \delta_{*},\nonumber
\end{align}
which along with \eqref{equ3-24''} entails that 
\begin{align*}
|V_{1}|\leq& \frac{\Lambda\delta_{*}}{2}, \quad |M_{1}|\leq \frac{\Lambda\delta_{*}}{2}.
\end{align*}
We further diminish $\delta_*$ smaller than $\frac{1}{3\Lambda}$ if necessary, and thereby obtain the bounds $|V_{1}|\leq1$ as well as $|M_{1}|\leq1$.
Combining this with \eqref{equ C_1+F_1} and \eqref{add3}, we verify that the conditions of Proposition \ref{pro3-1} are satisfied for the scaled 
quantities $(u_1,B_1,P_1)$. As a result, we get
\begin{align}\label{equ3-24'}
C_{r}\left[u_{1}-\overline{V}_{2},B_{1}-\overline{M}_{2},P_{1}-\overline{h}_{2}(t)\right]
\leq \frac{1}{2}C_{1}\left[u_{1},B_{1},P_{1}\right]+\frac{1}{2}\sqrt{F_{m,1}\left[f_{1},g_{1}\right]},
\end{align}
where the constant vectors $\overline{V}_{2}$, $\overline{M}_{2}$ and the function of time $\overline{h}_{2}(t)$ provided by Proposition \ref{pro3-1} satisfy
\begin{align}
|\overline{V}_{2}|\leq& \frac{\Lambda}{2}C_{1}\left[u_{1},B_{1},P_{1}\right]+\frac{\Lambda}{2}\sqrt{F_{m,1}\left[f_{1},g_{1}\right]},\label{{V}_{2}} \\
|\overline{M}_{2}|\leq& \frac{\Lambda}{2}C_{1}\left[u_{1},B_{1},P_{1}\right]+\frac{\Lambda}{2}\sqrt{F_{m,1}\left[f_{1},g_{1}\right]}.\label{{M}_{2}}
\end{align}
Using the definitions of $u_{1}$, $B_{1}$, $P_{1}$, $f_{1}$, $g_{1}$ from \eqref{equ3-24'}, and applying \eqref{equ C_1+F_1}, we derive
\begin{align}\label{equ3-25}
C_{r}\left[u_{1}-\overline{V}_{2},B_{1}-\overline{M}_{2},P_{1}-\overline{h}_{2}(t)\right]
& \leq \frac{1}{2}C_{1}\left[u_{1},B_{1},P_{1}\right]+\frac{1}{2} \sqrt{F_{m,1}\left[f_{1},g_{1}\right]}\nonumber\\
&\leq\frac{1}{2^{2}}C_{1}\left[u,B,P\right]+\frac{1}{2}\sqrt{F_{m,1}\left[f,g\right]}.
\end{align}
Calculations based on \eqref{scaled} imply that
the left-hand side of \eqref{equ3-25} is equivalently expressed as
\begin{align*}
C_{r}\left[u_{1}-\overline{V}_{2},B_{1}-\overline{M}_{2},P_{1}-\overline{h}_{2}(t)\right]
=C_{r^{2}}\left[u-V_{1}-\overline{V}_{2},B-M_{1}-\overline{M}_{2},P-h_{1}(t)-r^{-1}\overline{h}_{2}(t/r^{2})\right].
\end{align*}
With setting
\begin{align*}
V_{2}=V_{1}+\overline{V}_{2},\quad
M_{2}=M_{1}+\overline{M}_{2},\quad
h_{2}(t)=h_{1}(t)+r^{-1}\overline{h}_{2}(t/r^{2}),
\end{align*}
we conclude from \eqref{equ3-25} that 
\begin{align*}
C_{r^{2}}\left[u-V_{2},B-M_{2},P-h_{2}(t)\right]
\leq\frac{1}{2^{2}}C_{1}\left[u,B,P\right]+\frac{1}{2}\sqrt{F_{m,1}\left[f,g\right]}.
\end{align*}
In addition, $\eqref{{V}_{2}}$, $\eqref{{M}_{2}}$ and $\eqref{equ C_1+F_1}$ tell us
\begin{align*}
|{V}_{2}-{V}_{1}|\leq& \frac{\Lambda}{2}C_{1}\left[u_{1},B_{1},P_{1}\right]+\frac{\Lambda}{2}\sqrt{F_{m,1}\left[f_{1},g_{1}\right]}
\leq \frac{\Lambda}{2^{2}}C_{1}\left[u,B,P\right]+\frac{\Lambda}{2}\sqrt{F_{m,1}\left[f,g\right]},\\
|{M}_{2}-{M}_{1}|\leq& \frac{\Lambda}{2}C_{1}\left[u_{1},B_{1},P_{1}\right]+\frac{\Lambda}{2}\sqrt{F_{m,1}\left[f_{1},g_{1}\right]}
\leq \frac{\Lambda}{2^{2}}C_{1}\left[u,B,P\right]+\frac{\Lambda}{2}\sqrt{F_{m,1}\left[f,g\right]}.
\end{align*}
Combining the above inequalities with \eqref{equ3-24''} and \eqref{assump1} yields
\begin{align*}
|{V}_{2}|\leq& |{V}_{2}-{V}_{1}|+|{V}_{1}|\\
\leq& \frac{\Lambda}{2^{2}}C_{1}\left[u,B,P\right]+\frac{\Lambda}{2}\sqrt{F_{m,1}\left[f,g\right]}
+\frac{\Lambda}{2}C_{1}\left[u,B,P\right]+\frac{\Lambda}{2}\sqrt{F_{m,1}\left[f,g\right]}\\
=&\frac{3\Lambda}{4}C_{1}\left[u,B,P\right]+\Lambda\sqrt{F_{m,1}\left[f,g\right]}\\
\leq& \frac{3\Lambda}{4}\delta_{*}+\Lambda\delta_{*}\leq 1,
\end{align*}
where the final inequality follows from $\delta_*\leq\frac{1}{3\Lambda}$. An analogous argument shows $|{M}_{2}|\leq 1$.

After that, we target at the $k$-step iteration. Let us define
\begin{align*}
\left\{
  \begin{array}{ll}
u_{k}(x,t)=u(r^{k}x,r^{2k}t)-V_{k},\\
B_{k}(x,t)=B(r^{k}x,r^{2k}t)-M_{k},\\
P_{k}(x,t)=r^{k}P(r^{k}x,r^{2k}t)-r^{k}h_{k}(r^{2k}t)
     \end{array}
\right.
\end{align*}
and
\begin{align*}
\left\{
  \begin{array}{ll}
f_{k}(x,t)=r^{2k}f(r^{k}x,r^{2k}t),\\
g_{k}(x,t)=r^{2k}g(r^{k}x,r^{2k}t).
     \end{array}
\right.
\end{align*}
Then $(u_{k},B_{k},P_{k})$ is a suitable weak solution to the following equations:
\begin{align*}
\left\{
  \begin{array}{ll}
    \partial_{t}u_{k}-\Delta u_{k}+r^{k}(V_{k}+u_{k})\cdot\nabla u_{k}
    -r^{k}(M_{k}+B_{k})\cdot\nabla B_{k}+\nabla P_{k}=f_{k}               &\quad {\rm in}~Q_{1},\\
    \partial_{t}B_{k}-\Delta B_{k}+r^{k}(V_{k}+u_{k})\cdot\nabla B_{k}
   -r^{k}(M_{k}+B_{k})\cdot\nabla u_{k}=g_{k}                             &\quad {\rm in}~Q_{1},\\
    {\rm div}\,u_{k}=0,\quad {\rm div}\,B_{k}=0                             &\quad {\rm in}~Q_{1}.
     \end{array}
\right.
\end{align*}
Additionally, due to $r\le\lambda_0$, Lemma \ref{lem3-2} provides the decay estimate for the external force terms
\begin{align}\label{equ3-27-0}
\sqrt{F_{m,r^{i}}\left[f, g\right]}
\leq \frac{1}{2^{i}}\sqrt{F_{m,1}\left[f,g\right]},\quad \forall~ i\in \Bbb{N}^+.
\end{align}

Assume that for any $i$-step  with $1\leq i\leq k$, we have established the following iteration scheme:
\begin{align}\label{equ3-27}
\left\{
  \begin{array}{ll}
  C_{r^{i}}\left[u-V_{i}, B-M_{i},P-h_{i}(t)\right]
\leq \frac{1}{2^{i}}C_{1}\left[u,B,P\right]+\frac{i}{2^{i}}\sqrt{F_{m,1}\left[f,g\right]}, \\
|{V}_{i}-{V}_{i-1}|\leq\frac{\Lambda}{2}\Big(C_{1}\left[u_{i-1},B_{i-1},P_{i-1}\right]+\sqrt{F_{m,1}\left[f_{i-1},g_{i-1}\right]}\Big),\\
|{M}_{i}-{M}_{i-1}|\leq\frac{\Lambda}{2}\Big(C_{1}\left[u_{i-1},B_{i-1},P_{i-1}\right]+\sqrt{F_{m,1}\left[f_{i-1},g_{i-1}\right]}\Big).
     \end{array}
\right.
\end{align}
From the definitions of $u_i$, $B_i$, $P_i$, $f_i$ and $g_i$, together with \eqref{equ3-27}$_1$ and \eqref{equ3-27-0},
it follows that for $1\leq i\leq k$,
\begin{align}\label{add5}
C_{1}\left[u_{i},B_{i},P_{i}\right]+\sqrt{F_{m,1}\left[f_{i},g_{i}\right]}
=&C_{r^{i}}\left[u-V_{i}, B-M_{i},P-h_{i}(t)\right]+\sqrt{F_{m,r^{i}}\left[f, g\right]}\nonumber\\
\leq& \frac{1}{2^{i}}C_{1}\left[u,B,P\right]+\frac{i+1}{2^{i}}\sqrt{F_{m,1}\left[f,g\right]}.
\end{align}
We  deduce from \eqref{equ3-27}$_2$ that
\begin{align*}
|{V}_{k}|\leq& \sum^{k}_{i=2} |{V}_{i}-{V}_{i-1}|+|{V}_{1}|\nonumber\\
\leq& \frac{\Lambda}{2}\sum^{k-1}_{i=1}\Big(C_{1}\left[u_{i},B_{i},P_{i}\right]+\sqrt{F_{m,1}\left[f_{i},g_{i}\right]}\Big)
+\frac{\Lambda}{2}\Big(C_{1}\left[u,B,P\right]+\sqrt{F_{m,1}\left[f,g\right]}\Big).
\end{align*}
Thus, substituting \eqref{add5} with $1\le i\le k-1$ into the above inequality and utilizing \eqref{assump1} yields
\begin{align*}
|{V}_{k}|
\leq &\frac{\Lambda}{2}\sum^{k-1}_{i=1}\Big(\frac{1}{2^{i}}C_{1}\left[u,B,P\right]
+\frac{i+1}{2^{i}}\sqrt{F_{m,1}\left[f,g\right]}\Big)
+\frac{\Lambda}{2}\Big(C_{1}\left[u,B,P\right]+\sqrt{F_{m,1}\left[f,g\right]}\Big)\\
= &\Lambda\sum^{k-1}_{i=0}\Big(\frac{1}{2^{i+1}}C_{1}\left[u,B,P\right]
+\frac{i+1}{2^{i+1}}\sqrt{F_{m,1}\left[f,g\right]}\Big)\\
\leq &\Lambda\left(\delta_{*}+2\delta_{*}\right)\leq 1
\end{align*}
in view of the fact $\delta_*\le \frac{1}{3\Lambda}$. 
By similar reasonings, we also have $|{M}_{k}|\leq 1$.
Revisiting \eqref{add5} for $i=k$ and leveraging \eqref{assump1} again, 
we infer
\begin{align*}
C_{1}\left[u_{k}, B_{k},P_{k}\right]+\sqrt{F_{m,1}\left[f_{k}, g_{k}\right]}
\leq \delta_{*}\leq \delta_1,
\end{align*}
which verifies the smallness condition required in Proposition \ref{pro3-1}. 
Hence, we obtain vectors $\overline{V}_{k+1}$, $\overline{M}_{k+1}$ and a time-variable function $\overline{h}_{k+1}(t)$ 
such that  
\begin{align}\label{k+1}
C_{r}\left[u_{k}-\overline{V}_{k+1},B_{k}-\overline{M}_{k+1},P_{k}-\overline{h}_{k+1}(t)\right]
\leq \frac{1}{2}C_{1}\left[u_{k},B_{k},P_{k}\right]+\frac{1}{2}\sqrt{F_{m,1}\left[f_{k},g_{k}\right]},
\end{align}
where the constant vectors $\overline{V}_{k+1}$ and  $\overline{M}_{k+1}$ satisfy
\begin{align}
|\overline{V}_{k+1}|\leq& \frac{\Lambda}{2}C_{1}\left[u_{k},B_{k},P_{k}\right]+\frac{\Lambda}{2}\sqrt{F_{m,1}\left[f_{k},g_{k}\right]},\label{add4'}\\
|\overline{M}_{k+1}|\leq& \frac{\Lambda}{2}C_{1}\left[u_{k},B_{k},P_{k}\right]+\frac{\Lambda}{2}\sqrt{F_{m,1}\left[f_{k},g_{k}\right]}\label{add4}.
\end{align}
By using the definitions of $u_{k}$, $B_{k}$, $P_{k}$, $f_{k}$, $g_{k}$, and utilizing \eqref{k+1} and \eqref{add5}, it yields that
\begin{align}\label{equ3-28}
&C_{r^{k+1}}\left[u-V_{k}-\overline{V}_{k+1},B-M_{k}-\overline{M}_{k+1},P-h_{k}(t)-r^{-k}\overline{h}_{k+1}(t/r^{2k})\right] \nonumber\\
=&C_{r}\left[u_{k}-\overline{V}_{k+1},B_{k}-\overline{M}_{k+1},P_{k}-\overline{h}_{k+1}(t)\right] \nonumber\\
\leq& \frac{1}{2^{k+1}}C_{1}\left[u,B,P\right]+\frac{k+1}{2^{k+1}}\sqrt{F_{m,1}\left[f,g\right]}.
\end{align}
With assigning 
\begin{align*}
V_{k+1}=V_{k}+\overline{V}_{k+1},\quad
M_{k+1}=M_{k}+\overline{M}_{k+1},\quad
h_{k+1}(t)=h_{k}(t)+r^{-k}\overline{h}_{k+1}(t/r^{2k}),
\end{align*}
we immediately infer from \eqref{equ3-28} that
\begin{align*}
C_{r^{k+1}}\left[u-V_{k+1}, B-M_{k+1},P-h_{k+1}(t)\right]
\leq \frac{1}{2^{k+1}}C_{1}\left[u,B,P\right]+\frac{k+1}{2^{k+1}}\sqrt{F_{m,1}\left[f,g\right]},
\end{align*}
and derive from \eqref{add4'}--\eqref{add4} that
\begin{align*}
|V_{k+1}-V_{k}|\leq& \frac{\Lambda}{2}C_{1}\left[u_{k},B_{k},P_{k}\right]+\frac{\Lambda}{2}\sqrt{F_{m,1}\left[f_{k},g_{k}\right]},\\
|M_{k+1}-M_{k}|\leq& \frac{\Lambda}{2}C_{1}\left[u_{k},B_{k},P_{k}\right]+\frac{\Lambda}{2}\sqrt{F_{m,1}\left[f_{k},g_{k}\right]}.
\end{align*}
These estimates establish the validity of \eqref{equ3-27} for the $(k+1)$-step.

Finally, by applying the parabolic Campanato embedding theorem, there holds
\begin{align*}
C_{r^{k}}\left[u-\overline{u},B-\overline{B},0\right]
\leq C_{r^{k}}\left[u-V_{k},B-M_{k},0\right]
\leq\frac{1}{2^{k}}C_{1}\left[u,B,P\right]+\frac{k}{2^{k}}\sqrt{F_{m,1}\left[f,g\right]}
\leq\frac{k+1}{2^{k}}=r^{k\alpha},
\end{align*}
where 
\begin{align*}
\overline{u}=\dashint_{Q_{r^{k}}}u\,dxdt,\quad 
\overline{B}=\dashint_{Q_{r^{k}}}B\,dxdt.
\end{align*}
Therefore, we conclude that $u,B\in C^{\alpha}(\overline{Q}_{\frac{1}{2}})$
for some $\alpha\in(0,1)$, as expected.
\end{proof}

\section{Nonlinear Regularity with Small Energy}\label{sec4}

This section is dedicated to the proof of Theorem \ref{the1-2}. 
We first redefine the scaled solutions as follows:
\begin{align*}
u_{r}(x,t)=ru(rx,r^{2}t),\quad
B_{r}(x,t)=rB(rx,r^{2}t),\quad
P_{r}(x,t)=r^{2}P(rx,r^{2}t)
\end{align*}
with the corresponding scaled external forces:
\begin{align*}
f_{r}(x,t)=r^{3}f(rx,r^{2}t),\quad
g_{r}(x,t)=r^{3}g(rx,r^{2}t).
\end{align*}
The quantities associated with $(u,B,P)$ and $(f,g)$ are introduced as:
\begin{align*}
&N_{r}\left[u,B,P\right]:=rC_{r}\left[u,B,P\right]=r\left(\sqrt[3]{\dashint_{Q_{r}}|u|^{3}\,dxdt}
+\sqrt[3]{\dashint_{Q_{r}}|B|^{3}\,dxdt}
+r\sqrt[\frac{3}{2}]{\dashint_{Q_{r}}|P|^{\frac{3}{2}}\,dxdt}\right),\\
&Y_{m,r}\left[f,g\right]:=rF_{m,r}\left[f,g\right]=r^{3}\left(\sqrt[m]{\dashint_{Q_{r}}|f|^{m}\,dxdt}
+\sqrt[m]{\dashint_{Q_{r}}|g|^{m}\,dxdt}\right),\quad  m>\frac{5}{2}.
\end{align*}
Additionally, we define the local energy functional for the gradients:
\begin{align*}
E_{r}:=E_{r}\left[u,B\right]=\frac{1}{r}\int_{ Q_{r}}\left(|\nabla u|^{2}+|\nabla B|^{2}\right)dxdt,
\end{align*}
which admits the scaling invariance property
\begin{align*}
E_{r}\left[u,B\right]=E_{1}\left[u_{r},B_{r}\right],\quad \forall~r>0.
\end{align*}

The proof proceeds as follows. First, Lemmas \ref{lem4-1} and \ref{lem4-2} establish an iterative relation for $N_{r_0^k}\left[u,B,P\right]$. 
Then, invoking Lemma \ref{lem4-3}, we convert the preliminary assumption $\displaystyle\sup_{0 < r \leq 1} E_{r}\left[u,B\right] < \delta_{**}$ from Theorem \ref{the1-2} 
into the smallness condition required by Theorem \ref{the1-1}, after which the continuity of $(u,B)$ at the origin follows immediately.

\begin{lemma}\label{lem4-1}
Let $(u,B,P)$ be a suitable weak solution of \eqref{equ1-1} in $Q_{1}$ with $f,g\in L^{m}(Q_1)$ for $m>\frac{5}{2}$. 
Then there is a $C>0$ such that for $0<r\leq\frac{1}{2}$,
\begin{align}\label{equ4-1}
r^{3}\dashint_{Q_{r}}|u|^{3}\,dxdt
\leq CrN_{1}\left[u,B,P\right]^{3}+CrY_{m,1}\left[f,g\right]^{\frac{3}{2}}
+C\left(\frac{E_{1}^{2}}{r^{5}}
+\frac{ E_{1}^{\frac{4}{3}} }{r^{3}}
+\frac{ E_{1}^{\frac{6}{5}}  }{r^{\frac{13}{5}}}
+\frac{ E_{1}^{6}  }{r^{17}}\right),
\end{align}
\begin{align}\label{equ4-2}
r^{3}\dashint_{Q_{r}}|B|^{3}\,dxdt
\leq CrN_{1}\left[u,B,P\right]^{3}+CrY_{m,1}\left[f,g\right]^{\frac{3}{2}}
+C\left(\frac{E_{1}^{2}}{r^{5}}
+\frac{E_{1}^{\frac{4}{3}} }{r^{3}}
+\frac{E_{1}^{\frac{6}{5}} }{r^{\frac{13}{5}}}
+\frac{E_{1}^{6}}{r^{17}}\right),
\end{align}
\begin{align}\label{equ4-3}
r^{3}\dashint_{Q_{r}}|P|^{\frac{3}{2}}\,dxdt
\leq C\left(r+\frac{E_{1}}{r^{2}}\right)N_{1}\left[u,B,P\right]^{\frac{3}{2}}+\frac{CE_{1}}{r^{2}}Y_{m,1}\left[u,B,P\right]^{\frac{3}{4}}
+C\left(\frac{E_{1}^{2}}{r^{5}}+\frac{E_{1}^{\frac{3}{2}}}{r^{\frac{7}{2}}}\right).
\end{align}
\end{lemma}

\begin{proof}
We begin with the proof of \eqref{equ4-1} by applying \eqref{equ2-3}, which states that
\begin{align*}
\dashint_{B_{r}}|u|^{3}\,dx\leq C\dashint_{B_{1}}|u|^{3}\,dx
+\frac{C}{r^{3}}\Big(\int_{B_{\frac{3}{4}}}|u|^{2}\,dx\Big)^{\frac{1}{2}}\int_{B_{1}}|\nabla u|^{2}\,dx.
\end{align*}
Integrating this inequality w.r.t. the time-variable over $(-r^{2},0)\subset (-\frac{1}{4},0)$, we deduce that
\begin{align*}
\frac{1}{r^{3}|B_{1}|}\int_{Q_{r}}|u|^{3}\,dxdt
\leq C\dashint_{Q_{1}}|u|^{3}\,dxdt+\frac{C}{r^{3}}\sup_{t\in (-\frac{1}{4},0)}
\Big(\int_{B_{\frac{3}{4}}}|u|^{2}\,dx\Big)^{\frac{1}{2}}\int_{Q_{1}}|\nabla u|^{2}\,dxdt.
\end{align*}
Lemma \ref{lem2-5} applied to the above inequality implies that
\begin{align*}
\frac{1}{r^{3}|B_{1}|}\int_{Q_{r}}|u|^{3}\,dxdt
\leq& C\dashint_{Q_{1}}|u|^{3}\,dxdt
+\frac{C}{r^{3}}\left(\dashint_{Q_{1}}|u|^{3}+|B|^{3}+|P|^{\frac{3}{2}}\right)^{\frac{1}{2}}\int_{Q_{1}}|\nabla u|^{2}\,dxdt\\
&+\frac{C}{r^{3}}\left(\dashint_{Q_{1}}|u|^{3}+|B|^{3}\right)^{\frac{1}{6}}\int_{Q_{1}}|\nabla u|^{2}\,dxdt\\
&+\frac{C}{r^{3}}\left(\dashint_{Q_{1}}|f|^{m}+|g|^{m}\right)^{\frac{1}{2}}\int_{Q_{1}}|\nabla u|^{2}\,dxdt\\
\leq& CN_{1}^{3}+
\frac{C}{r^{3}}\left(N_{1}^{\frac{3}{2}}E_{1}
+N_{1}^{\frac{3}{4}}E_{1}
+N_{1}^{\frac{1}{2}}E_{1}
+Y_{m,1}^{\frac{m}{2}}E_{1}\right),
\end{align*}
where the last term can be estimated as 
\begin{align*}
Y_{m,1}^{\frac{m}{2}}E_{1}
=Y_{m,1}^{\frac{m}{2}-\frac{5}{4}}Y_{m,1}^{\frac{5}{4}}E_{1}
\le C Y_{m,1}^{\frac{5}{4}}E_{1}
\end{align*}
due to $m>\frac{5}{2}$ and $Y_{m,1}=\sqrt[m]{\dashint_{Q_{1}}|f|^{m}\,dxdt}
+\sqrt[m]{\dashint_{Q_{1}}|g|^{m}\,dxdt}<C$.
It follows from the Young inequality that
\begin{align*}
\frac{1}{r^{3}|B_{1}|}\int_{Q_{r}}|u|^{3}\,dxdt
\leq CN_{1}^{3}
+C\left(N_{1}^{3}+Y_{m,1}^{\frac{3}{2}}
+\frac{E_{1}^{2}}{r^{6}}
+\frac{E_{1}^{\frac{4}{3}}}{r^{4}}
+\frac{E_{1}^{\frac{6}{5}}}{r^{\frac{18}{5}}}
+\frac{E_{1}^{6}}{r^{18}}\right).
\end{align*}
The above display clearly yields our expected \eqref{equ4-1} after multiplying by $r$. 
The second item \eqref{equ4-2} can be obtained by the similar reasonings.

We now analyze the pressure term $P$, which satisfies the following equation
\begin{align*}
-\Delta P=&\sum_{i,j=1}^{3}\frac{\partial^{2}}{\partial x_{i}\partial x_{j}}\left(u^{i}u^{j}-B^{i}B^{j}\right)\\
=&\sum_{i,j=1}^{3}\frac{\partial^{2}}{\partial x_{i}\partial x_{j}}
\left((u^{i}-\overline{u}^{i})(u^{j}-\overline{u}^{j})-(B^{i}-\overline{B}^{i})(B^{j}-\overline{B}^{j})\right),
\end{align*}
where $\overline{u}^{i}=\dashint_{B_{\frac{3}{4}}}u^{i}\,dx$ and $\overline{B}^{i}=\dashint_{B_{\frac{3}{4}}}B^{i}\,dx$.
We utilize (b) in Lemma \ref{lem2-3} for $P$ with $p=\frac{3}{2}$ to find that
\begin{align}\label{equ4-51}
\dashint_{B_{r}}|P|^{\frac{3}{2}}\,dx
\leq C\dashint_{B_{1}}|P|^{\frac{3}{2}}\,dx
+\frac{C}{r^{3}}\dashint_{B_{\frac{3}{4}}}|u-\overline{u}|^{3}\,dx
+\frac{C}{r^{3}}\dashint_{B_{\frac{3}{4}}}|B-\overline{B}|^{3}\,dx.
\end{align}
With the help of the H\"{o}lder inequality, the Sobolev embedding and the Poincar\'{e} inequality, one sees
\begin{align}\label{equ4-52}
\frac{C}{r^{3}}\dashint_{B_{\frac{3}{4}}}|u-\overline{u}|^{3}\,dx
\leq&\frac{C}{r^{3}} \left(\int_{B_{\frac{3}{4}}}|u-\overline{u}|^{2}\,dx\right)^{\frac{3}{4}}
\left(\int_{B_{\frac{3}{4}}}|u-\bar{u}|^{6}\,dx\right)^{\frac{1}{4}} \nonumber\\
\leq&\frac{C}{r^{3}}\left(\int_{B_{\frac{3}{4}}}|u-\overline{u}|^{2}\,dx\right)^{\frac{3}{4}}
\left(\int_{B_{\frac{3}{4}}}|\nabla u|^{2}\,dx\right)^{\frac{3}{4}}\nonumber\\
\leq&\frac{C}{r^{3}}\left(\int_{B_{\frac{3}{4}}}|u-\overline{u}|^{2}\,dx\right)^{\frac{1}{2}}
\left(\int_{B_{\frac{3}{4}}}|\nabla u|^{2}\,dx\right)\nonumber \\
\leq&\frac{C}{r^{3}}\sup_{t\in (-\frac{1}{4},0)}\left(\int_{B_{\frac{3}{4}}}|u|^{2}\,dx\right)^{\frac{1}{2}}
\int_{B_{1}}|\nabla u|^{2}\,dx.
\end{align}
Similar arguments as above performed on $B$ give that
\begin{align}\label{equ4-53}
\frac{C}{r^{3}}\dashint_{B_{\frac{3}{4}}}|B-\overline{B}|^{3}\,dx
\leq\frac{C}{r^{3}}\sup_{t\in (-\frac{1}{4},0)}\left(\int_{B_{\frac{3}{4}}}|B|^{2}\,dx\right)^{\frac{1}{2}}
\int_{B_{1}}|\nabla B|^{2}\,dx.
\end{align}
Thus, by substituting \eqref{equ4-52} and \eqref{equ4-53} into \eqref{equ4-51}, there holds that
\begin{align*}
\dashint_{B_{r}}|P|^{\frac{3}{2}}\,dx
\leq& C\dashint_{B_{1}}|P|^{\frac{3}{2}}\,dx
+\frac{C}{r^{3}}\sup_{t\in (-\frac{1}{4},0)}\left(\int_{B_{\frac{3}{4}}}|u|^{2}+|B|^{2}\right)^{\frac{1}{2}}
\int_{B_{1}}\left(|\nabla u|^{2}+|\nabla B|^{2}\right)dx.
\end{align*}
An integration w.r.t. the time-variable over $(-r^{2},0)\subset (-\frac{1}{4},0)$ yields
\begin{align*}
\frac{1}{r^{3}|B_{1}|}\int_{Q_{r}}|P|^{\frac{3}{2}}\,dxdt
\leq C\dashint_{Q_{1}}|P|^{\frac{3}{2}}\,dxdt
+\frac{C}{r^{3}}\sup_{t\in (-\frac{1}{4},0)}\left(\int_{B_{\frac{3}{4}}}|u|^{2}+|B|^{2}\right)^{\frac{1}{2}}
\int_{Q_{1}}\left(|\nabla u|^{2}+|\nabla B|^{2}\right)dxdt.
\end{align*}
Lemma \ref{lem2-5} is used to infer that
\begin{align*}
\frac{1}{r^{3}|B_{1}|}\int_{Q_{r}}|P|^{\frac{3}{2}}\,dxdt
\leq& C\dashint_{Q_{1}}|P|^{\frac{3}{2}}\,dxdt
+\frac{C}{r^{3}}\left(\dashint_{Q_{1}}|u|^{3}+|B|^{3}+|P|^{\frac{3}{2}}\right)^{\frac{1}{2}}\int_{Q_{1}}\left(|\nabla u|^{2}+|\nabla B|^{2}\right)\,dxdt\\
&+\frac{C}{r^{3}}\left(\dashint_{Q_{1}}|u|^{3}+|B|^{3}\right)^{\frac{1}{6}}\int_{Q_{1}}\left(|\nabla u|^{2}+|\nabla B|^{2}\right)\,dxdt\\
&+\frac{C}{r^{3}}\left(\dashint_{Q_{1}}|f|^{m}+|g|^{m}\right)^{\frac{1}{2}}\int_{Q_{1}}\left(|\nabla u|^{2}+|\nabla B|^{2}\right)\, dxdt,
\end{align*}
where the definitions of $N_{1}$, $Y_{m,1}$, the boundedness of $Y_{m,1}$ and the Young inequality entail that
\begin{align*}
\frac{1}{r^{3}|B_{1}|}\int_{Q_{r}}|P|^{\frac{3}{2}}\,dxdt
\leq& CN_{1}^{\frac{3}{2}}
+\frac{C}{r^{3}}\left(N_{1}^{\frac{3}{2}}E_{1}
+N_{1}^{\frac{3}{4}}E_{1}
+N_{1}^{\frac{1}{2}}E_{1}
+Y_{m,1}^{\frac{m}{2}}E_{1}\right)\\
=& CN_{1}^{\frac{3}{2}}
+\frac{C}{r^{3}}\left(N_{1}^{\frac{3}{2}}E_{1}
+N_{1}^{\frac{3}{4}}E_{1}
+N_{1}^{\frac{1}{2}}E_{1}
+Y_{m,1}^{\frac{m}{2}-\frac{3}{4}}Y_{m,1}^{\frac{3}{4}}E_{1}\right)\\
\leq& CN_{1}^{\frac{3}{2}}
+\frac{C}{r^{3}}\left(N_{1}^{\frac{3}{2}}E_{1}
+N_{1}^{\frac{3}{4}}E_{1}
+N_{1}^{\frac{1}{2}}E_{1}
+Y_{m,1}^{\frac{3}{4}}E_{1}\right)\\
\leq& CN_{1}^{\frac{3}{2}}
+\frac{CE_{1}}{r^{3}}N_{1}^{\frac{3}{2}}
+\frac{CE_{1}}{r^{3}}Y_{m,1}^{\frac{3}{4}}
+C\left(N_{1}^{\frac{3}{2}}
+\frac{ E_{1}^{2}}{r^{6}}
+\frac{ E_{1}^{\frac{3}{2}}}{r^{\frac{9}{2}}}\right)\\
=& C\left(1+\frac{E_{1}}{r^{3}}\right)N_{1}^{\frac{3}{2}}
+\frac{CE_{1}}{r^{3}}Y_{m,1}^{\frac{3}{4}}
+C\left(\frac{E_{1}^{2}}{r^{6}}
+\frac{ E_{1}^{\frac{3}{2}}}{r^{\frac{9}{2}}}\right),
\end{align*}
which readily results in \eqref{equ4-3} after multiplying by $r$.
\end{proof}

\begin{lemma}\label{lem4-2}
Let $(u,B,P)$ be a suitable weak solution of \eqref{equ1-1} in $Q_{1}$ with $f,g\in L^{m}(Q_1)$ for $m>\frac{5}{2}$.
Take $\lambda_{0}=\frac{1}{8^{\frac{m}{2m-5}}}$ as given in Lemma \ref{lem3-2}.
There exist constants $0< r_0\leq\lambda_{0}$ and $0<\delta_{2}<1$ such that if $E_{\lambda}\left[u,B\right]\leq \delta_{2}$ for all $0<\lambda\leq1$, then
\begin{align}\label{equ4-5}
N_{r_0^{k}}\left[u,B,P\right]\leq \frac{1}{2^{k}}N_{1}\left[u,B,P\right]
+\frac{1}{2^{k}}\sum_{i=0}^{k-1}r_{0}^{\frac{i}{2}}\sqrt{Y_{m,1}\left[f,g\right]}
+\sum_{i=0}^{k-1}\frac{1}{2^{k-1-i}}E_{r_{0}^{i}}^{\frac{1}{3}}\left[u,B\right],\quad \forall~k\in\Bbb{N}^+.
\end{align}
\end{lemma}

\begin{proof}
By applying $(A+B+C)^{\alpha}\leq A^{\alpha}+B^{\alpha}+C^{\alpha}$ to \eqref{equ4-1} and \eqref{equ4-2} for $\alpha=\frac{1}{3}$
and to \eqref{equ4-3} for $\alpha=\frac{2}{3}$, and then summing the estimates, we obtain
\begin{align*}
N_{r}\leq& C\left(r^{\frac{1}{3}}+r^{\frac{2}{3}}
+\frac{E_{1}^{\frac{2}{3}}}{r^{\frac{4}{3}}}\right)\left(N_{1}+\sqrt{Y_{m,1}}\right)
+C\left(\frac{E_{1}^{\frac{2}{3}}}{r^{\frac{5}{3}}}
+\frac{E_{1}^{\frac{4}{9}}}{r}
+\frac{E_{1}^{\frac{2}{5}}}{r^{\frac{13}{15}}}
+\frac{E_{1}^{2}}{r^{\frac{17}{3}}}
+\frac{E_{1}^{\frac{4}{3}}}{r^{\frac{10}{3}}}
+\frac{E_{1}}{r^{\frac{7}{3}}}\right)\\
=& C\left(r^{\frac{1}{3}}+r^{\frac{2}{3}}
+\frac{E_{1}^{\frac{2}{3}}}{r^{\frac{4}{3}}}\right)\left(N_{1}+\sqrt{Y_{m,1}}\right)
+C\left(\frac{E_{1}^{\frac{1}{3}}}{r^{\frac{5}{3}}}
+\frac{E_{1}^{\frac{1}{9}}}{r}
+\frac{E_{1}^{\frac{1}{15}}}{r^{\frac{13}{15}}}
+\frac{E_{1}^{\frac{5}{3}}}{r^{\frac{17}{3}}}
+\frac{E_{1}}{r^{\frac{10}{3}}}
+\frac{E_{1}^{\frac{2}{3}}}{r^{\frac{7}{3}}}
\right)E_{1}^{\frac{1}{3}}.
\end{align*}
Hence, we select $r_0$ and $\delta_{2}$ sufficiently small such that
\begin{align*}
C\left(r_{0}^{\frac{1}{3}}+r_{0}^{\frac{2}{3}}+\frac{\delta_{2}^{\frac{2}{3}}}{r_{0}^{\frac{4}{3}}}\right)\leq \frac{1}{2}
\end{align*}
as well as
\begin{align*}
C\left(\frac{\delta_{2}^{\frac{1}{3}}}{r_{0}^{\frac{5}{3}}}
+\frac{\delta_{2}^{\frac{1}{9}}}{r_{0}}
+\frac{\delta_{2}^{\frac{1}{15}}}{r_{0}^{\frac{13}{15}}}
+\frac{\delta_{2}^{\frac{5}{3}}}{r_{0}^{\frac{17}{3}}}
+\frac{\delta_{2}}{r_{0}^{\frac{10}{3}}}
+\frac{\delta_{2}^{\frac{2}{3}}}{r_{0}^{\frac{7}{3}}}\right)
\leq 1,
\end{align*}
thereby ensuring \eqref{equ4-5} for the case of $k=1$.

According to the definition of $Y_{m,r_{0}^{k}}\left[f,g\right]$ and Lemma \ref{lem3-2}, it yields that
\begin{align}\label{equ4-6}
Y_{m,r_{0}^{k}}\left[f,g\right]=r_{0}^{k}F_{m,r_{0}^{k}}\left[f,g\right]
\leq\frac{r_{0}^{k}}{4^{k}}F_{m,1}\left[f,g\right]=\frac{r_{0}^{k}}{4^{k}}Y_{m,1}\left[f,g\right],\quad\forall~k\in\Bbb{N}^+.
\end{align}
We now prove \eqref{equ4-5} by induction. Assume that the following inequality holds for a given $k\in\Bbb{N}^+$:
\begin{align}\label{equ4-5'}
N_{r_{0}^{k}}\leq \frac{1}{2^{k}}N_{1}+\frac{1}{2^{k}}\sum_{i=0}^{k-1}r_{0}^{\frac{i}{2}}\sqrt{Y_{m,1}}+\sum_{i=0}^{k-1}\frac{1}{2^{k-1-i}}E_{r_{0}^{i}}^{\frac{1}{3}}.
\end{align}
Then, by applying \eqref{equ4-5} with $k=1$, \eqref{equ4-5'} and \eqref{equ4-6}, we derive
\begin{align*}
N_{r_{0}^{k+1}}\leq& \frac{1}{2}N_{r_{0}^{k}}+\frac{1}{2}\sqrt{Y_{m,r_{0}^{k}}}+E_{r_{0}^{k}}^{\frac{1}{3}}\\
\leq&\frac{1}{2}\left(\frac{1}{2^{k}}N_{1}+\frac{1}{2^{k}}\sum_{i=0}^{k-1}r_{0}^{\frac{i}{2}}\sqrt{Y_{m,1}}
+\sum_{i=0}^{k-1}\frac{1}{2^{k-1-i}}E_{r_{0}^{i}}^{\frac{1}{3}}\right)
+\frac{r_{0}^{\frac{k}{2}}}{2^{k+1}}\sqrt{Y_{m,1}}+E_{r_{0}^{k}}^{\frac{1}{3}}\\
=&\frac{1}{2^{k+1}}N_{1}+\frac{1}{2^{k+1}}\sum_{i=0}^{k}r_{0}^{\frac{i}{2}}\sqrt{Y_{m,1}}
+\sum_{i=0}^{k}\frac{1}{2^{k-i}}E_{r_{0}^{i}}^{\frac{1}{3}}.
\end{align*}
This completes the induction step, and therefore the entire proof.
\end{proof}

\begin{lemma}\label{lem4-3}
Assume that $(u,B,P)$ is a suitable weak solution of \eqref{equ1-1} in $Q_{1}$ with $f,g\in L^{m}(Q_1)$ for $m>\frac{5}{2}$.
Let $r_{0}>0$ and $\delta_{2}>0$ be as set forth in Lemma \ref{lem4-2}, and define $\delta_3:=\min\{\delta_{2},\frac{\delta_{*}^{3}}{2^{6}}\}$,
where $\delta_{*}$ is chosen according to Theorem \ref{the1-1}.
We have that if $E_{\lambda}\left[u,B\right]\leq\delta_3$ for all $0<\lambda\leq 1$, 
then $u$ and $B$ are H\"{o}lder continuous in $\overline{Q}_{r_{1}}$ for
\begin{align*}
r_{1}:=\frac{1}{2}\min\left\{r_{0}^{2}\left(\frac{N_{1}\left[u,B,P\right]}{\delta_{*}}\right)^{\frac{\ln r_{0}}{\ln 2}},
r_{0}^{4}\left(\frac{\sqrt{Y_{m,1}\left[f,g\right]}}{\delta_{*}}\right)^{\frac{\ln r_{0}}{\ln 2}}\right\}.
\end{align*}
\end{lemma}
\begin{proof}
By referring to equations \eqref{equ4-5} and \eqref{equ4-6}, we find
\begin{align*}
N_{r_0^{k}}\left[u,B,P\right]+\sqrt{Y_{m,r_0^{k}}\left[f,g\right]}
\leq& \frac{1}{2^{k}}N_{1}\left[u,B,P\right]+\frac{1}{2^{k}}\sum_{i=0}^{k}r_{0}^{\frac{i}{2}}\sqrt{Y_{m,1}\left[f,g\right]}
+\sum_{i=0}^{k-1}\frac{1}{2^{k-1-i}}E_{r_0^{i}}^{\frac{1}{3}}\left[u,B\right]\\
\leq&\frac{1}{2^{k}}N_{1}\left[u,B,P\right]+\frac{1}{2^{k-2}}\sqrt{Y_{m,1}\left[f,g\right]}
+2\delta_{3}^{\frac{1}{3}}, \quad\forall~ k\in\Bbb{N}^+,
\end{align*}
if we take
\begin{align*}
k\geq \max \left\{\log_{2}\left(\frac{N_{1}\left[u,B,P\right]}{\delta_{*}}\right)+2,
\log_{2}\left(\frac{\sqrt{Y_{m,1}\left[f,g\right]}}{\delta_{*}}\right)+4\right\},
\end{align*}
thus
\begin{align*}
N_{r_{0}^{k}}\left[u,B,P\right]+\sqrt{Y_{m,r_0^{k}}\left[f,g\right]}\leq \frac{\delta_{*}}{2^{2}}+\frac{\delta_{*}}{2^{2}}+\frac{\delta_{*}}{2}=\delta_{*}.
\end{align*}
Observing that the definitions of $u_{r_{0}^{k}}$, $B_{r_{0}^{k}}$, $P_{r_{0}^{k}}$ $f_{r_{0}^{k}}$ and $g_{r_{0}^{k}}$ ensure that
\begin{align*}
C_{1}\left[u_{r_{0}^{k}},B_{r_{0}^{k}},P_{r_{0}^{k}}\right]
+\sqrt{F_{m,1}\left[f_{r_{0}^{k}},g_{r_{0}^{k}}\right]}
=N_{r_{0}^{k}}\left[u,B,P\right]+\sqrt{Y_{m,r_0^{k}}\left[f,g\right]}\leq \delta_{*},
\end{align*}
we thus show that $u_{r_{0}^{k}}$ and $B_{r_{0}^{k}}$ are H\"{o}lder continuous in $\overline{Q}_{\frac{1}{2}}$ through the application of Theorem \ref{the1-1}.
This means that $u$ and $B$ are H\"{o}lder continuous in $\overline{Q}_{\frac{r_{0}^{k}}{2}}$, as claimed. 
\end{proof}

\begin{proof}
[\bf Proof of Theorem \ref{the1-2}.]
We let $r_{0}>0$, $\delta_{2}>0$ be as given in Lemma \ref{lem4-2} and $\delta_{*}$ be the constant from Theorem \ref{the1-1}, 
and then select $\delta_{**}:=\min\{\delta_{2},\frac{\delta_{*}^{3}}{2^{6}}\}$ in the statement of Theorem \ref{the1-2}.
For $R>0$, we notice that $E_{\frac{r}{R}}\left[u_R, B_R\right]= E_r\left[u,B\right]\le \delta_{**}$ for any $\frac{r}{R}\le 1$,
which implies the continuity of $u_R$ and $B_R$ in $Q_{r_1}$ by revisiting Lemma \ref{lem4-3}.
This guarantees that $u$ and $B$ are continuous in $Q_{ Rr_1 }$. 
The last conclusion regarding the one-dimensional parabolic Hausdorff measure of the singular set for any suitable 
weak solution follows from the Vitali covering argument, complete details can be found in the proof of {\cite[Proof of Theorem B]{CKN1982}.}
\end{proof}

In the final part, we shall utilize the arguments in \cite{WZ2013} to demonstrate that the velocity field still plays a more important role than the magnetic field, 
even in the presence of the external force term. 
To this end, let's  revisit the relevant quantities:
\begin{align*}
A_{r}\left[\varphi\right]&=\sup_{-r^{2}< t<0}r^{-1}\int_{B_{r}}|\varphi|^{2}\,dx,\quad\quad\quad\quad
E_{r}\left[\varphi\right]=r^{-1}\int_{Q_{r}}|\nabla \varphi|^{2}\,dxdt,\\
\tilde{C}_{r}\left[\varphi\right]&=r^{-2}\int_{Q_{r}}|\varphi|^{3}\,dxdt,\quad\quad\quad\quad\quad\quad\,\,\,
K_{r}\left[\varphi\right]=r^{-3}\int_{Q_{r}}|\varphi|^{2}\,dxdt.
\end{align*}
Let $A_{r}\left[u,B\right]=A_{r}\left[u\right]+A_{r}\left[B\right]$, and similarly $E_{r}\left[u,B\right]$, 
$\tilde{C}_{r}\left[u,B\right]$, $K_{r}\left[u,B\right]$ represent the corresponding expressions. 
Simultaneously, the following quantities are introduced:
\begin{align*}
G_{r}\left[\varphi,p,q\right]&=r^{1-\frac{3}{p}-\frac{2}{q}}\left\|\varphi\right\|_{L^{p,q}(Q_{r})},\quad\quad
\tilde{G}_{r}\left[\varphi,p,q\right]=r^{1-\frac{3}{p}-\frac{2}{q}}\left\|\varphi-(\overline{\varphi})_{B_{r}}\right\|_{L^{p,q}(Q_{r})},\\
H_{r}\left[\varphi,p,q\right]&=r^{2-\frac{3}{p}-\frac{2}{q}}\left\|\varphi\right\|_{L^{p,q}(Q_{r})},\quad\quad
\tilde{H}_{r}\left[\varphi,p,q\right]=r^{2-\frac{3}{p}-\frac{2}{q}}\left\|\varphi-(\overline{\varphi})_{B_{r}}\right\|_{L^{p,q}(Q_{r})}.
\end{align*}
Additionally, we define the quantity of the external force term as
\begin{align*}
M_{r}\left[\varphi\right]=r^{\frac{4m-10}{m}}\sqrt[\frac{m}{2}]{\int_{Q_{r}}|\varphi|^{m}\,dxdt},\quad m>\frac{5}{2}.
\end{align*}

\begin{proof}
  [\bf Proof of Corollary \ref{pro1-1}]
In fact, we only need to suppose that 
\begin{align*}
\sup_{0<r<\overline{R}}r^{1-\frac{3}{p}-\frac{2}{q}}\left\|u\right\|_{L^{p,q}(Q_{r})}<\overline{\delta}_{**}
\end{align*}
with $\overline{\delta}_{**}>0$ to be decided later and $(p,q)$ satisfying
\begin{align*}
\frac{3}{p}+\frac{2}{q} = 2,\quad 1\leq q\leq\infty,\quad (p,q)\neq (\infty,1).
\end{align*}

Let $0<4r<\rho<\overline{R}$, $1\leq p,q\leq\infty$, and $\zeta$ be a cutoff function which vanishes outside of $Q_{\rho}$, equaling 1 in $Q_{\rho/2}$ with
\begin{align*}
|\nabla\zeta |\leq C\rho^{-1},\quad  |\partial_{s}\zeta|, |\Delta\zeta|\leq C\rho^{-2} \quad {\rm in}~Q_{\rho}.
\end{align*}
After introducing the backward heat kernel as 
\begin{align*}
\Gamma (x,t)=\frac{1}{4\pi(r^{2}-t)^{3/2}}e^{-\frac{|x|^{2}}{4(r^{2}-t)}},
\end{align*}
we set the test function $\phi=\Gamma\zeta$ in the local energy inequality \eqref{equ1-2} to infer that
\begin{align*}
&\sup_{ -\rho^2< t<0 }\int_{B_{\rho}}\left(|u|^{2}+|B|^{2}\right)\phi\,dx
+2\int_{Q_{\rho}}\left(|\nabla u|^{2}+|\nabla B|^{2}\right)\phi\,dxds\\
\leq&\int_{Q_{\rho}}\Big(\left(|u|^{2}+|B|^{2}\right)\left(\partial_{s}\phi+\Delta\phi\right)
+u\cdot\nabla\phi\left(|u|^{2}+|B|^{2}+2P-(2P)_{B_{\rho}}\right)-2\left(u\cdot B\right)B\cdot\nabla\phi\Big)\,dxds\\
&+2\int_{Q_{\rho}}\left(u\cdot f+B\cdot g\right)\phi\,dxds\\
\leq& C\int_{Q_{\rho}}\Big(\left(|u|^{2}+|B|^{2}\right)\left(\Gamma\Delta\zeta+\Gamma\partial_{s}\zeta+2\nabla\Gamma\cdot\nabla\zeta\right)
+|\nabla\phi||u|\left(|u|^{2}+|B|^{2}+2|P-(P)_{B_{\rho}}|\right)\Big)\,dxds\\
&+C\int_{Q_{\rho}}\left(|u||f|+|B||g|\right)|\phi|\,dxds.
\end{align*}
The test function enjoys the following features:
\begin{align*}
&C^{-1}r^{-3} \leq \Gamma (x,t)\leq C r^{-3} \quad {\rm in}~ Q_{r};\\
&|\nabla\phi|\leq |\nabla\Gamma|\zeta+\Gamma|\nabla\zeta| \leq Cr^{-4},\quad
|\Gamma\Delta\zeta|+|\Gamma\partial_{s}\zeta|+2|\nabla\Gamma\cdot\nabla\zeta|\leq C\rho^{-5} \quad {\rm in}~Q_{\rho}.
\end{align*}
Let $(p',q')$ be the conjugate index of $(p,q)$. By applying the above properties of the test function and utilizing the H\"{o}lder inequality, we derive
\begin{align}\label{equ4-8}
&A_{r}\left[u,B\right]+E_{r}\left[u,B\right] \nonumber\\
\leq& C\left(\frac{r}{\rho}\right)^{2}K_{\rho}\left[u,B\right]
+C\left(\frac{\rho}{r}\right)^{2}\rho^{-2}
\int_{Q_{\rho}}\left(|u|^{3}+|u||B|^{2}+|u||P-(P)_{B_{\rho}}|\right)\,dxds \nonumber\\
&+C\left(\frac{r}{\rho}\right)^{2}\rho^{-1}\int_{Q_{\rho}}\left(|u||f|+|B||g|\right)\,dxds\nonumber\\
\leq& C\left(\frac{r}{\rho}\right)^{2}K_{\rho}\left[u,B\right]+C\left(\frac{\rho}{r}\right)^{2}
\left(\tilde{C}_{\rho}\left[u\right]+G_{\rho}\left[u,p,q\right]\left(G_{\rho}\left[B,2p',2q'\right]^{2}
+\tilde{H}_{\rho}\left[P,p',q'\right]\right)\right) \nonumber\\
&+C\left(\frac{r}{\rho}\right)^{2}\left(A_{\rho}\left[u,B\right]+M_{\rho}\left[f,g\right]\right), 
\end{align}
where we have used the fact 
\begin{align*}
\rho^{-1}\int_{Q_{\rho}}|u||f|\,dxds
\leq& \rho^{-1}\int_{Q_{\rho}}\left(|u|^{2}+|f|^{2}\right)dxds\\
\leq& \rho\sup_{-\rho^{2}< t<0}\int_{B_{\rho}}|u|^{2}\, dx
+\rho^{-1}\left(\int_{Q_{\rho}}|f|^{2\times\frac{m}{2}}\,dxds\right)^{\frac{2}{m}}\left(\int_{Q_{\rho}}1\,dxds\right)^{1-\frac{2}{m}}\\
\leq& A_{\rho}\left[u\right]+M_{\rho}\left[f\right],
\end{align*}
and similarly,
\begin{align*}
\rho^{-1}\int_{Q_{\rho}}|B||g|\,dxds\leq A_{\rho}\left[B\right]+M_{\rho}\left[g\right].
\end{align*}

To proceed our proof, we need to recall some existing estimates for the quantities appearing in \eqref{equ4-8}. As stated in \cite[Lemmas 3.1--3.3]{WZ2013}, we have
\begin{align}\label{equ4-9}
\tilde{C}_{r}\left[u\right]\leq C\overline{\delta}_{**}\left(A_{r}\left[u\right]+E_{r}\left[u\right]\right),
\end{align}
\begin{align}\label{equ4-10}
G_{r}\left[B,2p',2q'\right]^{2}\leq C\left(A_{r}\left[B\right]+E_{r}\left[B\right]\right),
\end{align}
\begin{align}\label{equ4-11}
\tilde{H}_{r}\left[P,p',q'\right]
\leq C\left(\frac{\rho}{r}\right)\tilde{G}_{\rho}\left[u,B,2p',2q'\right]^{2}
+C\left(\frac{r}{\rho}\right)^{3/p'}\tilde{H}_{\rho}\left[P,p',q'\right].
\end{align}
Let us choose $0<8r<\rho<\kappa<\overline{R}$ and set
\begin{align}\label{equ4-12}
\tilde{F}\left[r\right]=A_{r}\left[u,B\right]+E_{r}\left[u,B\right]
+\overline{\delta}_{**}^{1/2}\tilde{H}_{r}\left[P,p',q'\right]+M_{r}\left[f,g\right].
\end{align}
Combining \eqref{equ4-8}--\eqref{equ4-12}, it follows that
\begin{align*}
\tilde{F}\left[r\right]\leq& C\left(\frac{r}{\rho}\right)^{2}K_{\rho}\left[u,B\right]
+C\left(\frac{\rho}{r}\right)^{2}\left(\tilde{C}_{\rho}\left[u\right]
+\overline{\delta}_{**} G_{\rho}\left[B,2p',2q'\right]^{2}
+\overline{\delta}_{**}\tilde{H}_{\rho}\left[P,p',q'\right]\right)\\
&+C\left(\frac{r}{\rho}\right)^{2}\tilde{F}\left[\rho\right]
+\overline{\delta}_{**}^{1/2}\tilde{H}_{r}\left[P,p',q'\right]+M_{r}\left[f,g\right]\\
\leq& C\left(\frac{r}{\rho}\right)^{2}\tilde{F}\left[\rho\right]
+C\left(\frac{\rho}{r}\right)^{2}\overline{\delta}_{**}\tilde{F}\left[\rho\right]
+C\left(\frac{\rho}{r}\right)^{2}\overline{\delta}_{**}\tilde{H}_{\rho}\left[P,p',q'\right]
+\overline{\delta}_{**}^{1/2}\tilde{H}_{r}\left[P,p',q'\right]+M_{r}\left[f,g\right].
\end{align*}
Taking $r=\theta\rho$ and $\rho=\theta\kappa$ with $0<\theta<\frac{1}{8}$, we get
\begin{align*}
\tilde{F}\left[r\right]\leq& C\left(\frac{r}{\rho}\right)^{2}\tilde{F}\left[\rho\right]
+C\left(\frac{\rho}{r}\right)^{2}\overline{\delta}_{**}\tilde{F}\left[\rho\right]
+C\left(\frac{\rho}{r}\right)^{2}\overline{\delta}_{**}\left(\frac{\kappa}{\rho}\right)\tilde{F}\left[\kappa\right]\\
&+C\left(\frac{\rho}{r}\right)^{2}\overline{\delta}_{**}\left(\frac{\rho}{\kappa}\right)^{3/p'}\tilde{H}_{\kappa}\left[P,p',q'\right]
+C\overline{\delta}_{**}^{1/2}\left(\frac{\kappa}{r}\right)\tilde{F}\left[\kappa\right]\\
&+C\overline{\delta}_{**}^{1/2}\left(\frac{r}{\kappa}\right)^{3/p'}\tilde{H}_{\kappa}\left[P,p',q'\right]
+C\left(\frac{r}{\kappa}\right)^{(4m-10)/m}\tilde{F}\left[\kappa\right]\\
\leq& C\bigg(\left(\frac{r}{\rho}\right)^{2}\left(\frac{\kappa}{\rho}\right)
+\left(\frac{\rho}{r}\right)^{2}\overline{\delta}_{**}\left(\frac{\kappa}{\rho}\right)
+\left(\frac{\rho}{r}\right)^{2}\overline{\delta}_{**}^{1/2}\left(\frac{\rho}{\kappa}\right)^{3/p'}\\
&+\overline{\delta}_{**}^{1/2}\left(\frac{\kappa}{r}\right)+\left(\frac{r}{\kappa}\right)^{3/p'}
+\left(\frac{r}{\kappa}\right)^{(4m-10)/m}\bigg)\tilde{F}\left[\kappa\right]\\
\leq& C\left(\theta+\overline{\delta}_{**}\theta^{-3}+\overline{\delta}_{**}^{1/2}\theta^{-2+3/p'}
+\overline{\delta}_{**}^{1/2}\theta^{-2}+\theta^{6/p'}+\theta^{(8m-20)/m}\right)\tilde{F}\left[\kappa\right].
\end{align*}
If we select $\theta$ sufficiently small and then choose $\overline{\delta}_{**}$ small enough such that
\begin{align*}
C\left(\theta+\overline{\delta}_{**}\theta^{-3}+\overline{\delta}_{**}^{1/2}\theta^{-2+3/p'}
+\overline{\delta}_{**}^{1/2}\theta^{-2}+\theta^{6/p'}+\theta^{(8m-20)/m}\right) \leq \frac{1}{2},
\end{align*}
then the following iterative inequality can be obtained
\begin{align*}
\tilde{F}\left[\theta^{2}\kappa\right]\leq \frac{1}{2}\tilde{F}\left[\kappa\right].
\end{align*}
Besides, it is straightforward to observe that
\begin{align*}
\tilde{F}\left[R\right]\leq C_{1},
\end{align*}
where $C_{1}>0$ relies on $R$, $\left\|(u,B)\right\|_{L^\infty((-1,0);L^{2}(B_{1}))\cap L^2((-1,0);H^{1}(B_{1}))}$ 
and $\left\|(f,g)\right\|_{L^m((-1,0);L^{m}(B_{1}))}$. 
Therefore, there exists $r_{*}=r_{*}(C_1,\delta_{**})\in(0,\overline{R})$ derived from standard iteration arguments such that
\begin{align*}
\tilde{F}\left[r\right]\leq \delta_{**},\quad\forall~ 0<r<r_{*}.
\end{align*}
This result indicates that
\begin{align*}
\frac{1}{r}\int_{Q_{r}}\left(|\nabla u|^{2}+|\nabla B|^{2}\right) dxdt\leq \delta_{**},\quad \forall~0<r<r_{*},
\end{align*}
which enables us to apply Theorem \ref{the1-2} and subsequently determine the local H\"{o}lder continuity.
It means that the velocity field plays a more significant role than the magnetic field in the local regularity theory of the MHD equations, 
which is consistent with the findings in \cite{WZ2013}.
\end{proof}

\subsection*{Acknowledgements} This work was supported by the National Natural Science Foundation of China (Nos. 12471128 and 12401250).

\subsection*{Conflict of interest} The authors declare that there is no conflict of interest. We also declare that this
manuscript has no associated data.

\subsection*{Data availability} Data sharing is not applicable to this article as no datasets were generated or analysed
during the current study.

 \end{document}